\pdfoutput=1
\RequirePackage{ifpdf}
\ifpdf % We are running pdfTeX in pdf mode
\documentclass[pdftex]{sigma}
\else
\documentclass{sigma}
\fi

\newcommand{\Z}{{\mathbb Z}} %Integers
\newcommand{\C}{{\mathbb C}} %Complex line
\renewcommand{\O}{{\mathcal O}} %Holomorphic functions
\newcommand{\CP}[1]{\mathbb{C}P^{#1}} %complex projective spaces
\newcommand{\E}{{\mathcal E}} %Underlying holomorphic bundle
\newcommand{\Et}[1]{\widetilde{E}_{#1}} %affine Kac--Moody
\newcommand{\F}{{\mathcal F}} %Higgs sub-bundle
\renewcommand{\t}{\theta} %Original Higgs field
\newcommand{\Mod}{{\mathcal M}} %Moduli space
\renewcommand{\d}{{\rm d}} %Differential
\newcommand{\gl}{\mathfrak{gl}} %Lie algebra gl(r,C)
\newcommand{\CPbar}{{\overline {{\mathbb {C}}P}}^2}
\newcommand{\pardeg}{\deg_{\vec{\alpha}}}
\DeclareMathOperator{\res}{res}
\DeclareMathOperator{\Gl}{GL}

\numberwithin{equation}{section}
\newtheorem{prop}{Proposition}[section]
\newtheorem{lem}[prop]{Lemma}
\newtheorem{cor}[prop]{Corollary}
\newtheorem{thm}[prop]{Theorem}
\newtheorem{assn}[prop]{Assumption}
{ \theoremstyle{definition}
\newtheorem{defn}[prop]{Definition}
\newtheorem{Example}[prop]{Example}
\newtheorem{Remark}[prop]{Remark}}

\begin{document}
\allowdisplaybreaks

\newcommand{\arXivNumber}{1808.10125}

\renewcommand{\PaperNumber}{085}

\FirstPageHeading

\ShortArticleName{Hitchin Fibrations with One Singular Fiber}

\ArticleName{Hitchin Fibrations on Two-Dimensional Moduli\\ Spaces of Irregular Higgs Bundles with One Singular\\ Fiber}

\Author{P\'eter IVANICS~$^\dag$, Andr\'as I.~STIPSICZ~$^\dag$ and Szil\'ard SZAB\'O~$^\ddag$}

\AuthorNameForHeading{P.~Ivanics, A.I.~Stipsicz and S.~Szab\'o}

\Address{$^\dag$~R\'enyi Institute of Mathematics, 1053 Budapest, Re\'altanoda utca 13-15, Hungary}
\EmailD{\href{mailto:ipe@math.bme.hu}{ipe@math.bme.hu}, \href{mailto:stipsicz.andras@renyi.hu}{stipsicz.andras@renyi.hu}}
\Address{$^\ddag$~Budapest University of Technology and Economics, \\
\hphantom{$^\ddag$}~1111 Budapest, Egry J\'ozsef utca~1, H \'ep\"ulet, Hungary}
\EmailD{\href{mailto:szabosz@math.bme.hu}{szabosz@math.bme.hu}}

\ArticleDates{Received February 02, 2019, in final form October 25, 2019; Published online November 04, 2019}

\Abstract{We analyze and completely describe the four cases when the Hitchin fibration on a $2$-dimensional moduli space of irregular Higgs bundles over $\CP1$ has a single singular fiber. The case when the fiber at infinity is of type $I_0^*$ is further analyzed, and we give constructions of all the possible configurations of singular curves in
elliptic fibrations having this type of singular fiber at infinity.}

\Keywords{irregular Higgs bundles; Hitchin fibration; elliptic fibrations}

\Classification{14J27; 14H40; 14H60}

\section{Introduction}\label{sec:intro}
Let $D$ be an effective divisor in $\CP1$ of total length $4$. Consider the (complex) 2-dimensional moduli spaces ${\overline{\Mod^{\vec{\alpha}-s}\big(\CP1, D, 2, d\big)}}$ of rank~$2$ {degree~$d$} parabolically stable Higgs bundles over~$\CP{1}$ with irregular singularities {of prescribed local form at $D$ and of $0$ parabolic degree}. When equipped with the Hitchin fibration, this space has been shown to be biregular to the complement of a singular fiber in an elliptic fibration on a rational elliptic surface (which complex surface is diffeomorphic to the 9-fold blow-up $\CP{2}\# _9 \CPbar$ of the complex projective plane), see \cite[Section~3.3]{ISS2} for the details of biregularity.

In the following we will describe all the cases
when the Hitchin fibration on the moduli space has a single singular
fiber. In these cases the fibration on the compactified space has
exactly two singular fibers; fibrations with exactly two singular
fibers have been extensively studied~\cite{TwoSing1, TwoSing3}~--
indeed, the genus-1 case (which is of central relevance in our
subsequent studies) admits a simple classification scheme given in
\cite[Theorem~3.2]{TwoSing3}. In \cite[Theorem~3.2]{TwoSing3} four
cases and an infinite family of genus-1 fibrations with exactly two
singular fibers have been encountered -- since by
\cite[Section~4]{Sz-spectral} all fibrations we will study arise from
pencils (hence admit a section), the infinite family will not arise in
our context.

We will distinguish five main cases, depending on the number of poles
of the Higgs field (constrained by our assumption having a complex
2-dimensional moduli space):
\begin{enumerate}\itemsep=0pt
\item[1)] there is one pole (which has to be of multiplicity 4),
\item[2)] there are two poles, each of multiplicity 2,
\item[3)] there are two poles, one with multiplicity 3 and the other with
multiplicity 1,
\item[4)] there are three poles, one with multiplicity 2 and the further two with multiplicity 1, and finally
\item[5)] there are four poles, each with multiplicity 1. %\label{item:P6}
\end{enumerate}
At each pole we need to distinguish further two cases, depending on
whether the leading-order term of the polar part of the Higgs field at
the pole is a regular semi-simple endomorphism (untwisted case), or
has non-vanishing nilpotent part (twisted case). {The third
 possible case, when the polar part is non-regular semi-simple, does
 not provide elliptic fibration, as shown by a simple
 calculation similar to \cite[Lemma~5.6]{ISS2}.}

The polar parts will depend on some complex parameters, and our {first} aim is
to determine those parameter values for which the Hitchin fibration on
the moduli space has a unique singular fiber. These parameter values
are complex numbers once we fix certain trivializations~-- for
details see Section~\ref{sec:prelims}.
The second aim of the paper is to provide explicit constructions for all the
possibilities of configurations of singular fibers in case~(5) of
the above list (four logarithmic poles).

The spaces considered in this paper have a rich geometry. In
 addition to the Dolbeault algebraic structure that we focus on, they
 admit a de Rham algebraic structure parameterizing certain
 connections with irregular singularities (related to the Dolbeault
 space via non-abelian Hodge theory), and a Betti algebraic structure
 parameterizing representations of the fundamental group of the
 punctured surface together with Stokes data (related to de Rham
 space via the irregular Riemann--Hilbert correspondence). This
 picture was first studied by Hitchin~\cite{Hit_selfduality} for
 regular objects, then by Simpson~\cite{Sim} in the tame case,
 subsequently extended to wild singularities by~Biquard
 and~Boalch~\cite{BB}. In a series of papers,~Boalch studied the
 moduli spaces of connections with irregular singularities on a curve
 and the corresponding wild character varieties. In particular,
 Boalch~\cite{Boalch1} (in the untwisted case) and~Boalch
 and~Yamakawa~\cite{BY} (in the twisted case) defined the notion of
 an irregular curve, and showed that for a fixed irregular curve,
 parabolic weights and orbits of the residues in appropriate sense,
 one gets a Betti moduli functor (first spelled out in~\cite{Boalch3}
 in the untwisted case) whose corresponding coarse moduli space
 exists, and is given by a quasi-Hamiltonian quotient
 construction~\cite{AMM}. In the particular cases under study here,
 the above results imply that the moduli spaces we are looking at
 admit the structure of a smooth hyperK\"ahler manifold of real
 dimension $4$. It is known~\cite{Boalch2, CB, KraftProcesi} that an open
 subset of each such space is in fact a quiver
 variety~\cite{KN}. More precisely, it is a fiber of the
 semi-universal deformation of a Kleinian surface singularity, with
 an action of the corresponding Weyl group on the base of the
 deformation family. Another equivalent way of labelling these spaces
 that appears in the literature is by the type of the central fiber
 of this family.

The results of this paper, together with those of the previous
 papers~\cite{ISS, ISS2} of the same \mbox{authors}, will be
 applied~\cite{Sz_PW} by the third author to prove the numerical
 version of the \mbox{$P=W$} conjecture by~de Cataldo, Hausel
 and~Migliorini~\cite{dCHM} and Simpson's geometric counterpart
 thereof~\cite{Sim_boundary} in the particular cases covered by our
 papers, namely those corresponding to the Painlev\'e systems. It
 turns out that these conjectures are easier to treat in the presence
 of a $\C^{\times}$-action. In the Painlev\'e cases (and more
 generally, if the moduli space is complex $2$-dimensional), the
 existence of such an action implies that the Hitchin fibration
 admits a unique singular fiber. Therefore, one may hope that finding
 moduli spaces with a unique singular Hitchin fiber may be of use in
 the generalization of the $P=W$ conjecture.

Throughout, we will be interested in moduli spaces $\Mod^{\vec{\alpha}-s} \big(\CP1, D, 2\big)$ of rank $2$ parabolic Higgs
bundles over $\CP1$ with higher order poles at some divisor $D$, with prescribed local form near the points of the
divisor. For details, see Section~\ref{ssec:param_geo}. The moduli space admits a natural decomposition
\begin{gather*}
 \Mod^{\vec{\alpha}-s} \big(\CP1, D, 2\big) = \coprod_{d\in\Z} \Mod^{\vec{\alpha}-s} \big(\CP1, D, 2, d\big)
\end{gather*}
according to the degree of the underlying vector bundle $\E$. In the cases where the data fixed to define $\Mod^{\vec{\alpha}-s} \big(\CP1, D, 2, d\big)$
contains a regular non-semisimple (i.e., non-closed) adjoint orbit at some logarithmic points $t\in D$, we will equally study the partial compactification
\begin{gather}\label{eq:compactification}
 \overline{\Mod^{\vec{\alpha}-s} \big(\CP1, D, 2, d\big)}
\end{gather}
parameterizing parabolic Higgs bundles with possibly non-regular residues at the points $t$, having the same characteristic polynomial.
This functor differs from $\Mod^{\vec{\alpha}-s} \big(\CP1, D, 2, d\big)$ in that instead of imposing the adjoint orbit of the residue of the Higgs field
at such points $t$, we merely impose that the residue lie in the \emph{closure} of the given non-closed orbit.
Then, \eqref{eq:compactification} is a stratified algebraic space: the
top-dimensional stratum is $\Mod^{\vec{\alpha}-s} \big(\CP1, D, 2, d\big)$, and
the lower-dimensional strata are moduli spaces of parabolic
logarithmic Higgs bundles with some regular and some non-regular
residues having the same eigenvalues as fixed for the definition of
$\Mod^{\vec{\alpha}-s} \big(\CP1, D, 2,d\big)$.

\subsection{The case of one pole}\label{ss:OnePole}
This case was treated in \cite{ISS}, where we gave a complete
classification of singular fibers of the Hitchin fibration.
In particular, in the untwisted case the fibration depends on
the complex parameters
\begin{gather}\label{eq:untwisted-parameters}
 a_{\pm}, b_{\pm}, c_{\pm}, \lambda_{\pm} \in \C ,\qquad a_+ \neq a_- ,\tag{{U}}
 \end{gather}
and the fiber at infinity is of type $\Et7$. {The $a_+ = a_-$ case does
not lead to an elliptic fibration, cf.~\cite[Lemma~5.6]{ISS2}.}
For the geometric meaning of these parameters, see Section~\ref{sssec:4};
the various singular fibers are defined in Section~\ref{sec:elliptic}.
Let
\begin{gather*} \Delta = \big((b_--b_+){}^2-4 (a_--a_+) (c_--c_+)\big){}^3-432 (a_--a_+){}^4 \lambda _+^2.\end{gather*}
With these notations we have
\begin{thm}[{\cite[Theorem~1.1]{ISS}}]\label{thm:OnePoleUntwisted}
In this case the Hitchin fibration on the moduli space $\Mod ^{\vec{\alpha}-s}{\big(\CP1, D, 2, d\big)}$ admits a single singular fiber if
and only if $\lambda _+=\lambda _-=0$ and $\Delta =0$. The single singular fiber is of type $III$.
\end{thm}

In the twisted case we have complex parameters (described in detail Section~\ref{sssec:4})
\begin{gather}\label{eq:twisted-parameters}
 b_{-8},\ldots ,b_{-3} \in \C ,\qquad b_{-7} \neq 0, \tag{{T}}
 \end{gather}
the fiber at infinity is $\Et8$. The case $b_{-7} = 0$ does not give elliptic fibration \cite[Section~2.3.2]{ISS}.
We introduce
${\Gamma}=\big(b_{-6}^2+4 b_{-5}\big){}^2-24 b_{-7} (b_{-6} b_{-4}+2 b_{-3})$ and we have
\begin{thm}[{\cite[Theorem~1.3]{ISS}}]\label{thm:OnePoleTwisted}
In this case the Hitchin fibration on the moduli space $\Mod
^{\vec{\alpha}-s}{\big(\CP1, D, 2, d\big)}$ admits a single singular fiber
if and only if ${\Gamma} =0$. The single singular fiber is of
type $II$ $($i.e., a cusp fiber$)$.
\end{thm}

\subsection{The case of two poles}\label{ss:TwoPoles}

Let us consider first the case when both multiplicities are equal to~2. In this case a simple argument (see also
\cite[Theorems~{2.1,~2.2,~2.3}]{ISS2}) shows that the fiber at infinity
is $I_n^*$ with $n\in \{2,3,4\}$ (depending on whether the leading
order term is regular semi-simple or not at the two poles), hence in
this case there is no Hitchin fibration on $\Mod ^{\vec{\alpha}-s}{(\CP1, D, 2, d)}$ with exactly one
singular fiber (cf.\ Corollary~\ref{cor:ins}).

In the case of multiplicities 3 and 1 on the two poles, we have more
examples; indeed, if the pole with multiplicity 3 is
untwisted and the other one is twisted, then the fibration
depends on the complex parameters
$a_+$, $a_-$, $b_+$, $b_-$, $\lambda_+$, $\lambda_-$, $b_{-1}$;
their geometric significance is given in Section~\ref{sssec:31}.
{We assume $a_+\neq a_-$ because $a_+= a_-$ does not lead elliptic fibration see \cite[Remark~6.7]{ISS2}.}
{In addition, the fibration also depends on so-called parabolic
weights $\vec{\alpha} = \{ (\alpha_j^+, \alpha_j^-) \}_{j=1,2}$.}
By taking
\[
A=a_--a_+,\qquad B=b_--b_+,\qquad L=\lambda_--\lambda_+,
\]
and defining ${\Theta} = 4 A^2 \big(B^2-6 A L\big)$, we get

\begin{thm}[{\cite[Theorem~{2.5}]{ISS2}, \cite[Lemma~2]{Sz_PW}}]\label{thm:TwoPolesUT}
Suppose that we have two poles, one with multiplicity~$3$ and the other
one with multiplicity~$1$. Assume furthermore that the multiplicity~$3$
pole is untwisted, while the other one is twisted. $($In this case the
fiber at infinity is of type~$\Et6.)$ Then the Hitchin fibration
on $\overline{\Mod^{\vec{\alpha}-s} \big(\CP1, D, 2, d\big)}$ has
a~unique singular fiber if and only if ${\Theta} =0$ and $L=0$.
For generic parabolic weights, the single singular fiber is of type~$IV$.
\end{thm}

If the multiplicity 1 pole is untwisted, then according to
\cite[Theorems~{2.4,~2.6}]{ISS2} we do not have examples for unique
singular fibers in the Hitchin fibration on the moduli space.
Finally, if both poles are twisted, then
the natural parameters (described in Section~\ref{sssec:31})
take the form
$b_{-6}$, $b_{-5}$, $b_{-4}$, $b_{-3}$, $b_{-2}$, $b_{-1}$.
Define $Q=8b_{-5}$ ({$Q\neq 0$ due to \cite[Lemma~6.8]{ISS2}} and $R=b_{-4}^2+4 b_{-3}$;
with these notations in place, we have
\begin{thm}[{\cite[Theorem~{2.7}]{ISS2}, \cite[Lemma~2]{Sz_PW}}]\label{thm:TwoPolesTT}
Suppose that the multiplicities of the poles are~$3$ and~$1$, and both
are twisted. In this case the fiber at infinity is of type
$\Et7$, and the Hitchin fibration on $\overline{\Mod^{\vec{\alpha}-s} \big(\CP1, D, 2, d\big)}$
admits a single singular fiber if and only if $R=0$.
{For generic parabolic weights}, the single singular fiber is of type $III$.
\end{thm}

\subsection{The case of three poles}\label{ss:ThreePoles}
In this case one of the poles is of multiplicity 2 and the two others
are both of multiplicity~1. A~simple argument shows that the
fiber at infinity is either $I_1^*$ (when the polar part is untwisted
at the pole of multiplicity~2) or $I_2^*$
(when the polar part is twisted at the pole of multiplicity 2).
Since by \cite[Section~6]{SSS} elliptic fibrations on the rational elliptic
surface with a singular fiber $I_1^*$ or $I_2^*$ have at least three
singular fibers (cf.\ Corollary~\ref{cor:ins}), we conclude

\begin{thm}\label{thm:ThreePoles}
In the case of three poles, the Hitchin fibration admits at least two
singular fibers on the moduli space ${\overline{\Mod ^{\vec{\alpha}-s}}}$.
\end{thm}

\subsection{The case of four poles}\label{ss:FourPoles}
This is the most interesting case. It is not hard to see, that the
fiber at infinity is an $I_0^*$ fiber, and (once again by
\cite[Section~6]{SSS}) there is one case when the complement of this
fiber admits a single singular fiber: when this other fiber is also of
type $I_0^*$.

Before turning to the statement, we would like to point out a new
feature in this case. In all previous cases the holomorphic structure
on the fiber at infinity (which was either~$\Et8$,~$\Et7$ or~$\Et6$) was
unique. The holomorphic structure on~$I_0^*$, however is not
unique. All five $\CP{1}$-components of~$I_0^*$ can be equipped with a
unique holomorphic structure, but the four points where the curves
corresponding to the leaves of the graph describing~$I_0^*$ (see
Fig.~\ref{fig:regi}) intersect the central rational curve determine
a cross ratio, which is a holomorphic invariant.
We can assume that the four points of intersections are at
$0$, $1$, $\infty$ and $t\in \CP{1}$; hence $t$ is the complex parameter
(distinct from $0$, $1$, $\infty$) determining the complex structure on~$I_0^*$. The pencil {on the Hirzebruch surface} in this case is determined by four complex numbers
$a$, $b$, $c$, $d$ (cf.\ Section~\ref{sssec:1111} and equations~\eqref{eq:abcd} for their geometric role).
The fibration again depends on a choice of parabolic weights
$\vec{\alpha} = \{ (\alpha_j^+, \alpha_j^-) \}_{j=1,\ldots ,4}$, see Section~\ref{s:parabolic}.
We assume that the sum of all these weights is an integer, and moreover that
this integer is equal to the negative of the degree of the underlying vector
bundle of the Higgs bundle. With these notations and assumptions, we have

\begin{thm}\label{thm:FourPoles}The pencil given by the above parameters $t$, $a$, $b$, $c$, $d$ provides an
elliptic fibration with a single singular fiber $($besides the~$I_0^*$ fiber at infinity$)$ if and only if
\begin{gather*} c= a+(b-a-d)t+dt^2.\end{gather*} For generic parabolic weights the single singular fiber of
the partial compactification \begin{gather*} \overline{\Mod^{\vec{\alpha}-s} \big(\CP1, D, 2,d\big)}\end{gather*} introduced in~\eqref{eq:compactification} is of type $I_0^*$.
\end{thm}

The singular fibers of Hitchin fibrations in cases of one or two poles
have been completely determined in \cite{ISS, ISS2} in terms of the
parameters. In principle a similar analysis can be carried out in the
remaining cases -- in Section~\ref{sec:app} we will restrict
ourselves to provide examples of all possible singular fiber
configurations when we have four poles. For the list of these
possibilities, see Table~\ref{tab:singfibers} in
Section~\ref{sec:app}. Here we do not determine the various
configurations in terms of the parameters. By providing various
constructions and examples, we prove:

\begin{thm}\label{thm:I0*osszeseset}
Any configuration of singular fibers on the complement of an $I_0^*$ fiber
in a rational elliptic surface arises as the set of singular fibers of a
Hitchin fibration.
\end{thm}

The paper is organized as follows. In Section~\ref{sec:prelims} we collect some facts and definitions of Higgs bundles and the Hitchin
map, relevant to our present investigation. We also include the necessary stability analysis here.
Section~\ref{sec:elliptic} is devoted to list background material
regarding elliptic fibrations, and in Section~\ref{sec:proofs} we
describe the proofs of the statements about Hitchin fibrations with unique singular fibers given in this introduction.
Finally in Section~\ref{sec:app} we describe a~number of constructions for elliptic fibrations with one of the
singular fibers being $I_0^*$ (relevant in the case of four poles
in the above discussion), hence proving Theorem~\ref{thm:I0*osszeseset}. In these constructions we pay special
attention to the further singular fibers of the resulting
fibrations.

\section{Preliminaries}\label{sec:prelims}

In this paper we will deal with certain meromorphic Higgs bundles over $\CP1$ with structure group $\Gl(2,\C)$.
In the sequel we let $D$ be an effective divisor over $\CP1$ and denote by $K_{\CP1} = \mathbf{\Omega}_{\CP1}^1$ the canonical bundle of $\CP1$.

\begin{defn}A {\it meromorphic Higgs bundle with polar divisor $D$} (or a~{\it $K(D)$-pair of rank $2$}) over $\CP1$ consists of a pair $(\E,
\theta)$ where $\E$ is a holomorphic rank $2$ vector bundle over~$\CP1$ and
\begin{gather*}
 \theta\colon \ \E \to \E \otimes K_{\CP1}(D)
\end{gather*}
is an $\O_{\CP1}$-linear vector bundle morphism, called the {\it Higgs field}.
\end{defn}

From now on, we assume that the total length of $D$ is equal to $4$:
\begin{gather}\label{eq:divisor}
 D = m_1 t_1 + \cdots + m_k t_k, \qquad m_1 + \cdots + m_k = 4,
\end{gather}
where $1 \leq k \leq 4$, $1\leq m_j$ and $t_j \in \CP1$ are distinct points for $1 \leq j \leq k$.
We have an $\O_{\CP1}$-linear isomorphism of line bundles
\begin{gather}\label{eq:KDO2}
 K_{\CP1}(D) \cong \O_{\CP1}(2).
\end{gather}

We denote
\begin{itemize}\itemsep=0pt
 \item {by ${\mathbb {F}}_2$ the Hirzebruch surface obtained
 by the projectivization of $K_{\CP1}(D)$, and by $p\colon {\mathbb
 {F}}_2 \to \CP1$ the ruling coming from this construction,}
 \item by $\O_{{\mathbb {F}}_2 | \CP1}(1)$ the relative ample bundle
 of $p$ and by $\O_{{\mathbb {F}}_2 | \CP1}(n) = \O_{{\mathbb {F}}_2
 | \CP1}(1)^{\otimes n}$ its tensor powers,
 \item by $\zeta \in H^0 \big({\mathbb {F}}_2 ,p^* K_{\CP1}(D)
 \otimes \O_{{\mathbb {F}}_2 | \CP1}(1) \big)$ and $\xi \in H^0
 \big({\mathbb {F}}_2 , \O_{{\mathbb {F}}_2 | \CP1}(1) \big)$ the
 canonical sections,
 \item and by $\mbox{I}_{\E}$ the identity endomorphism of $\E$.
\end{itemize}

Let $(\E, \theta)$ be a $K(D)$-pair of rank $2$ over $\CP1$.
\begin{defn}
The {\it characteristic polynomial} of $\theta$ is the section
\begin{gather}\label{eq:char-poly1}
 \chi_{\theta} (\zeta ) = \det (\zeta \mbox{I}_{\E} - \xi p^* \theta) = \zeta^2 + \zeta \xi F_{\theta} + \xi^2 G_{\theta} \in H^0 \big({\mathbb {F}}_2 ,p^* K_{\CP1}(D)^{\otimes 2} \otimes \O_{{\mathbb {F}}_2 | \CP1}(2) \big).\!\!\!
\end{gather}
Here
 \begin{gather*}
 F_{\theta}\in H^0\big(\CP1 , K_{\CP1}(D) \big), \qquad G_{\theta} \in H^0\big(\CP1 , K_{\CP1}(D)^{\otimes 2} \big)
 \end{gather*}
 are called the {\it characteristic coefficients} of $\theta$. Elements of $H^0\big(\CP1 , K_{\CP1}(D)\big)$ will be called {\it meromorphic differentials} and
 elements of $H^0\big(\CP1 , K_{\CP1}(D)^{\otimes 2}\big)$ {\it meromorphic quadratic dif\-fe\-rentials}.
 The {\it spectral curve} of $(\E , \theta )$ is the curve in ${\mathbb {F}}_2$ defined by the equation $\chi_{\theta} (\zeta ) = 0$.
 The {\it Hitchin base} is the vector space
 \begin{gather*}
 \mathcal{B} = H^0\big(\CP1 , K_{\CP1}(D) \big) \oplus H^0\big(\CP1 , K_{\CP1}(D)^{\otimes 2} \big)
 \end{gather*}
 containing all possible characteristic coefficients of rank $2$ $K(D)$-pairs.
\end{defn}
It follows from the definition and \eqref{eq:KDO2} that the
characteristic coefficients $F_{\theta}$, $G_{\theta}$ may be considered
as global sections of $\O_{\CP1}(2)$ and of $\O_{\CP1}(4)$ respectively, so we have
\begin{gather*}
 \dim_{\C} \mathcal{B} = 8.
\end{gather*}
\begin{Remark}
In Section~\ref{sec:app} we will take a converse approach by setting $F_{\theta} = 0$ and considering sections
\begin{gather*}
 \sigma \in H^0 \big(\CP1 , \O_{\CP1} (4) \big)
\end{gather*}
(corresponding to the coefficient $G_{\theta}$); we then choose the points $t_j$ (or equivalently, the fibers $F_j = p^{-1}(t_j)$) appropriately in order to obtain various singular fibers.
\end{Remark}

Let us denote by
\begin{gather*}
 \Mod \big(\CP1, D, 2\big)
\end{gather*}
the moduli stack of $K(D)$-pairs of rank $2$ over $\CP1$. It decomposes as
\begin{gather*}%\label{eq:modulistack}
 \Mod \big(\CP1, D, 2\big) = \coprod_{d\in\Z} \Mod \big(\CP1, D, 2, d\big)
\end{gather*}
where $\Mod \big(\CP1, D, 2, d\big)$ is the moduli stack of $K(D)$-pairs of rank $2$ over $\CP1$ satisfying $\deg (\E )\allowbreak = d$.
\begin{defn}\label{def:Hitchin_map}
 The {\it Hitchin map} is the morphism
 \begin{align*}
 h\colon \ \Mod \big(\CP1, D, 2\big) & \to \mathcal{B}, \\
 (\E ,\theta ) & \mapsto (F_{\theta}, G_{\theta})
 \end{align*}
 associating to a Higgs bundle its characteristic coefficients.
\end{defn}
For generic $(F,G) \in \mathcal{B}$, the curve $\Sigma_{(F,G)}$ defined by the equation
\begin{gather*}
 \zeta^2 + \zeta \xi F + \xi^2 G
\end{gather*}
is smooth, and for every $d\in \Z$ the fiber of
\begin{gather}\label{eq:Hitchin-map}
 h\colon \ \Mod \big(\CP1, D, 2, d\big) \to \mathcal{B}
\end{gather}
is a torsor over the Jacobian $\operatorname{Jac} (\Sigma_{(F,G)})$ of
$\Sigma_{(F,G)}$. When the curve $\Sigma_{(F,G)}$ is not regular, then
$\operatorname{Jac} (\Sigma_{(F,G)})$ is non-compact, and for every $d$ the fiber
$h^{-1} (F,G)$ is (non-canonically) isomorphic to a certain compactification
of $\operatorname{Jac} (\Sigma_{(F,G)})$. In case $\Sigma_{(F,G)}$ is integral,
the suitable compactifi\-cation of $\operatorname{Jac} (\Sigma_{(F,G)})$ is the
moduli space of torsion-free sheaves of rank $1$ and given degree over~$\Sigma_{(F,G)}$. On the other hand, in case $\Sigma_{(F,G)}$ is
reducible or non-reduced, the corresponding fiber of~\eqref{eq:Hitchin-map} is not
of finite type, and in order to get a moduli scheme of finite type one
needs to impose a stability condition on the objects. Stability
conditions naturally arise from parabolic structures on meromorphic
Higgs bundles; we will return to describing moduli spaces of stable
objects and the fibers of the Hitchin map on these moduli spaces once
we have the assumptions and notations concerning the local
behaviour of a meromorphic Higgs field near the points of~$D$.

\subsection{Parameters and their geometric significance}\label{ssec:param_geo}
Our investigation is based on the local description of spectral
curves on the Hirzebruch surface~{${\mathbb {F}}_2$} of degree~2.
Recall from equation~\eqref{eq:divisor} that
$D = m_1 t_1 + \cdots + m_k t_k,$ with $ m_1 + \cdots + m_k = 4$.

If we are in the case of one pole with multiplicity $4$ {at
 $t_1$}, then we introduce two local charts on $\CP{1}$: $U$ with $u\in
\C$ (where $\{u=0\}=[0:1]=t_1$) and $V$ with $v \in \C$ (where $\{
v=0\}=[1:0]$). The bundle $K_{\CP{1}}(4t_1)$ admits the trivializing
sections $\kappa_U$ over $U$ and $\kappa_V$ over $V$:
\begin{align*}
 \kappa_U = \frac{\d u}{u^4}, \qquad 	\kappa_V = \d v.
\end{align*}

If we are in the case of at least two poles, then we introduce
 two local charts on $\CP{1}$: $U$~with $u\in \C$ (where
 $\{u=0\}=[0:1]=t_1$) and $V$ with $v \in \C$ (where $\{
 v=0\}=[1:0]=t_k$). We also denote by $t_i$ the representation
	of a pole in the local chart $U$, we refer $t_1$ as $0$, $t_k$ as
	$\infty$ and~$t_i^{-1}$ as the image of $t_i\in\CP{1}$ in chart $V$.
	The bundle $K_{\CP{1}}(D)$ admits the following
 trivializing sections:
\begin{gather*}
 \kappa_U = \frac{\d u}{\prod\limits_{i=2}^{k-1}(u-t_i)^{m_i}}, \qquad
	\kappa_V = \frac{\d v}{\prod\limits_{i=2}^{k-1}\big(v-t_i^{-1}\big)^{m_i}}.
\end{gather*}

The conversion from $\kappa_U$ to $\kappa_V$ is the following (if the total multiplicity is $4$):
\begin{gather*}
	{\kappa_U=-\frac{v^2}{\prod\limits_{i=2}^{k-1}(-t_i)^{m_i}} \kappa_V.}
\end{gather*}
The trivialization $\kappa_j$ induces a trivialization $\kappa_j^2$ on~$K_{\CP{1}}(D)^{\otimes 2}$, $j=U,V$.
Moreover, (recalling that $\zeta$ denotes the canonical section of $p^* K_{\CP1}(D) \otimes \O_{{\mathbb {F}}_2 | \CP1}(1)$) we can introduce
coordinates~$w$ in $p^{-1}(U)$ and $z$ in $p^{-1}(V)$ by
\begin{gather}	\label{eq:w-coord}
	\zeta=w\otimes \kappa_U=z\otimes \kappa_V.
\end{gather}

Choose a suitable trivialization of $\E$.
Consider an irregular Higgs bundle $(\E, \theta)$ in the $\kappa_j$
trivializations ($j=U,V$) and the chosen trivialization of $\E$.
The local forms of $\theta$ near $t_i$'s
are the following:
\begin{itemize}\itemsep=0pt
\item the case of $D=4 t_1$:
 \begin{gather}	\label{eq:4_theta}
 \theta = \sum_{n\geq -4} A_n u^n \otimes \d u,
 \end{gather}

\item the case of $D=2t_1+2t_2$ is not of interest, {because we
 concentrate on the cases with two singular fibers, but according to
 \cite[Theorems~2.1,~2.2,~2.3]{ISS2}, $D=2t_1+2t_2$ to leads at least
 three singular fibers,}

\item the case of $D=3t_1+1t_2$:
\begin{gather}	\label{eq:31_theta}
	\theta = \sum_{n\geq -3} A_n u^n \otimes \d u \qquad \text{and} \qquad \theta = \sum_{n\geq -1} B_n v^n \otimes \d v,
\end{gather}

\item the case of $D=2t_1+1t_2+1t_3$ is not relevant: a simple
 geometric argument similar to Example~\ref{exam:F2} shows that in
 this case the elliptic fibration contains a fiber of type $I_1^*$ or
 $I_2^*$, and according to Corollary~\ref{cor:ins} this case can not
 lead to exactly two singular fibers,

\item the case of $D=1t_1+1t_2+1t_3+1t_4$. Near one of the $t_i$ ($i=1,2,3$) the local form is
\begin{subequations}	\label{eq:1111_theta}
	\begin{gather}	%\label{eq:1111_theta_1}
	\theta = \sum_{n\geq -1} A_{n,i} (u-t_i)^n \otimes \d u
\end{gather}
and near $t_4=\infty$:
\begin{gather}	%\label{eq:1111_theta_2}
	\theta= \sum_{n\geq -1} A_{n,4} v^n \otimes \d v,
	\end{gather}
\end{subequations}
where $A_{n}, B_n, A_{n,i} \in \gl (2,\C )$.
\end{itemize}

The coefficient $A_{-1,i}$ in above expression is called the
{\it residue} of $\theta$ at $t_i$, with respect to
the chosen trivialization of~$\E$. The residue of $\theta$ at~$t_i$
is denoted by $\res_{t_i} (\theta)$. It is a well-defined endomorphism
of~$\E|_{t_i}$ ($i=1,\dots,4$).

The spectral curve ($\chi_{\theta} (\zeta ) = 0$) has an expansion
near $t_i$ in which these parameters have a~geometric
meaning. The lowest index matrices~$A_n$,~$B_n$ and $A_{n,i}$ in equations~\eqref{eq:4_theta},~\eqref{eq:31_theta} and~\eqref{eq:1111_theta}
encode the base locus of a pencil.

Let us consider the local forms separated according to the number of poles and their multiplicities.

\subsubsection[$D=4 t_1$]{$\boldsymbol{D=4 t_1}$}\label{sssec:4}
The matrices $A_{-3}$, $A_{-2}$ and $A_{-1}$ in equation~\eqref{eq:4_theta} encode the tangent, the
second and third derivative of the curve in the pencil.
If the leading order term is a regular semi-simple endomorphism (untwisted case) then the local form in the trivialization $\kappa_U$
{of $K_{\CP1}$ and the chosen trivialization of $\E$} is
\begin{gather*}
 \t = \left[ \begin{pmatrix}
 a_+ & 0 \\
 0 & a_-
 \end{pmatrix} u^{-4}
				+
				\begin{pmatrix}
 b_+ & 0 \\
 0 & b_-
 \end{pmatrix} u^{-3}
				+
				\begin{pmatrix}
 c_+ & 0 \\
 0 & c_-
 \end{pmatrix} u^{-2}
				+
				\begin{pmatrix}
 \lambda_+ & 0 \\
 0 & \lambda_-
 \end{pmatrix} u^{-1}
				+ O(1)
 \right] \otimes \d u.
\end{gather*}
These matrices provide the parameters~\eqref{eq:untwisted-parameters} in Theorem~\ref{thm:OnePoleUntwisted}.

If the leading order term has non-vanishing nilpotent part (twisted
case) then the local form in the trivialization $\kappa_U$ is
\begin{gather*}
 \t = \left[ \begin{pmatrix}
 b_{-8}\! & 1 \\
 0	& \!b_{-8}
 \end{pmatrix} u^{-4}\!
 +
 \begin{pmatrix}
 0 & 0 \\
 b_{-7}\! & \!b_{-6}
 \end{pmatrix} u^{-3}\!
 +
 \begin{pmatrix}
 0 & 0 \\
 b_{-5}\! & \!b_{-4}
 \end{pmatrix} u^{-2}\!
 +
 \begin{pmatrix}
 0 & 0 \\
 b_{-3}\! & \!b_{-2}
 \end{pmatrix} u^{-1}
 + O(1)
 \right] \otimes \d u,
 \end{gather*}
with $b_{-7} \neq 0$. These matrices provide the parameters~\eqref{eq:twisted-parameters} in Theorem~\ref{thm:OnePoleTwisted}.

\subsubsection[$D=3t_1+1t_2$]{$\boldsymbol{D=3t_1+1t_2}$}\label{sssec:31}
The matrices $A_{-2}$ and $A_{-1}$ in equation~\eqref{eq:31_theta}
encode the slope of the tangent line of a pencil and the second
derivative of the curve in the pencil.

If the pole with multiplicity $3$ (i.e., $t_1$) is untwisted then the local form
in the trivializa\-tion~$\kappa_U$ (near $t_1$) and the chosen trivialization of~$\E$:
\begin{gather*}
 \t = \left[ \begin{pmatrix}
 a_+ & 0 \\
 0 & a_-
 \end{pmatrix} u^{-3}
				+
				\begin{pmatrix}
 b_+ & 0 \\
 0 & b_-
 \end{pmatrix} u^{-2}
				+
				\begin{pmatrix}
 \lambda_+ & 0 \\
 0 & \lambda_-
 \end{pmatrix} u^{-1}
				+ O(1)
 \right] \otimes \d u.
\end{gather*}
If the multiplicity $3$ pole is twisted, then
\begin{gather*}
 \t = \left[ \begin{pmatrix}
 b_{-6} & 1 \\
 0	& b_{-6}
 \end{pmatrix} u^{-3}
 +
 \begin{pmatrix}
 0 & 0 \\
 b_{-5} & b_{-4}
 \end{pmatrix} u^{-2}
 +
 \begin{pmatrix}
 0 & 0 \\
 b_{-3} & b_{-2}
 \end{pmatrix} u^{-1}
 + O(1)
 \right] \otimes \d u,
\end{gather*}
{with $b_{-5} \neq 0$.}

The matrix $B_{-1}$ contains the base locus of the pencil in the
trivialization $\kappa_V$ and the chosen trivialization of $\E$. If the pole $t_2$ is untwisted, then
\begin{gather*}
 \t = \left[ \begin{pmatrix}
 \mu_+ & 0 \\
 0 & \mu_-
 \end{pmatrix} v^{-1} + O(1)
 \right] \otimes \d v.
 \end{gather*}
If the pole $t_2$ is twisted, then up to holomorphic gauge transformation
\begin{gather*}
 \t = \left[ \begin{pmatrix}
 b_{-1} & \varepsilon \\
 0	& b_{-1}
 \end{pmatrix} v^{-1}
 + O(1)
 \right] \otimes \d v
\end{gather*}
for $\varepsilon \in \C$.
These parameters appear in Theorems~\ref{thm:TwoPolesUT} and~\ref{thm:TwoPolesTT}.

\subsubsection[$D=1t_1+1t_2+1t_3+1t_4$]{$\boldsymbol{D=1t_1+1t_2+1t_3+1t_4}$}\label{sssec:1111}
The matrices in equations~\eqref{eq:1111_theta} only encode the base
locus of the pencil in the four distinguished fibers at $t_1=0$, $t_2=1$,
$t_3=t$ and $t_4=\infty \in \CP{1}$.

If one of the poles $t_i$ ($i=1,\dots, 4$) is untwisted then the
matrix from the local form in the trivialization $\kappa_j$ ($j=U,V$) and
the chosen trivialization of $\E$ is the following:
\begin{gather}\label{eq:4poles_ut_u}
 A_{-1,i} = \begin{pmatrix}
 a_{+,i} & 0 \\
 0 & a_{-,i}
 \end{pmatrix},
\end{gather}
where we assume that $a_{+,i}\neq a_{-,i}$, otherwise a similar
 computation as in \cite[Lemma~5.6]{ISS2} shows that the fibration is
 not elliptic.

Similarly, if one of the poles $t_i$ ($i=1,\dots, 4$) is twisted then the
matrix has non-trivial nilpotent part. In the trivialization $\kappa_j$ ($j=U,V$) and the chosen trivialization of $\E$:
\begin{gather}\label{eq:4poles_t_u}
 A_{-1,i} = \begin{pmatrix}
 a_{i} & \varepsilon \\
 0 & a_{i}
 \end{pmatrix}
 \end{gather}
for $\varepsilon \in \C$.

For simplicity we introduce new notations for the parameters appearing above:
\begin{subequations}\label{eq:abcd}
\begin{gather}
a_\pm:= a_{\pm,1}, \\
b_\pm:= a_{\pm,2}, \\
c_\pm:= a_{\pm,3}, \\
d_\pm:= a_{\pm,4}, \\
a:= a_{1}, \\
b:= a_{2}, \\
c:= a_{3}, \\
d:= a_{4}.
\end{gather}
\end{subequations}

The parameters $a$, $b$, $c$ and $d$ which appear in the Theorem~\ref{thm:FourPoles} encode the base locus in the distinguished
fibers.

\section{Higgs bundles with non-reduced spectral curve}\label{s:parabolic}

Let $(\E , \theta)$ be a meromorphic Higgs bundle of rank $2$ over
$\CP1$ with divisor $D$. Parabolic stability for Higgs fields with
higher-order poles has been investigated in
\cite{ISS, ISS2}. In this subsection we assume that the multiplicity
$m_j$ is equal to $1$ for all $j$, i.e., that
\begin{gather*}
 D = t_1 + \cdots + t_4
\end{gather*}
for some distinct points $t_j$ ($j=1,2,3,4$).
\begin{defn}
 A {\it compatible quasi-parabolic structure} is the choice of a $1$-dimensional eigen\-space $\ell_j \subset \E|_{t_j}$ of $\res_{t_j} (\theta )$ for all $1 \leq j \leq 4$.
 A {\it compatible parabolic structure} consists of a~compatible quasi-parabolic structure and a choice of a pair of rational numbers $(\alpha_j^+, \alpha_j^-)$
 (called {\it parabolic weights}) satisfying
 \begin{gather*}
 0 \leq \alpha_j^- < \alpha_j^+ < 1
 \end{gather*}
 for each $j$. The {\it parabolic degree} of $\E$ is defined by the formula
 \begin{gather*}
 \pardeg (\E ) = \deg (\E ) + \sum_{j=1}^4 (\alpha_j^+ + \alpha_j^-).
 \end{gather*}
 A {\it rank $1$ meromorphic Higgs subbundle} of $(\E ,\theta )$ is a pair $(\F , \theta|_{\F})$ where $\F$ is a rank $1$ subbundle of $\E$ such that the restriction
 $\theta|_{\F}$ maps $\F$ into $\F \otimes K_{\CP1}(D)$.
 \end{defn}

Let $\big(\E , \theta, \{ \ell_j \}_{j=1}^4\big)$ be a meromorphic Higgs
 bundle of rank $2$ over $\CP1$ with a compatible quasi-parabolic
 structure, and with given parabolic weights $\{ (\alpha_j^+,
 \alpha_j^-) \}_{j=1}^4$. For any rank $1$ meromorphic Higgs
 subbundle $(\F , \theta|_{\F})$ of $(\E ,\theta )$, the fiber
 $\F|_{t_j}$ is preserved by $\res_{t_j}(\theta )$. If $\F|_{t_j} = \ell_j$, then we set
 \begin{gather*}%\label{eq:weightF}
 \alpha_j(\F ) = \alpha_j^+,
 \end{gather*}
 otherwise we set
 \begin{gather*}
 \alpha_j(\F ) = \alpha_j^-.
 \end{gather*}
 Finally, we define
 \begin{gather*}
 \pardeg (\F ) = \deg (\F ) + \sum_{j=1}^4 \alpha_j(\F ).
 \end{gather*}

 \begin{defn} $\big(\E ,\theta , \{ \ell_j \}_{j=1}^4\big)$ is {\it $\vec{\alpha}$-stable} if for any rank $1$ meromorphic Higgs subbundle $(\F , \theta|_{\F})$ we have
 \begin{gather*}
 \pardeg (\F ) < \frac{\pardeg (\E )}2.
 \end{gather*}
 Moreover, $\big(\E ,\theta ,\{ \ell_j \}_{j=1}^4\big)$ is {\it $\vec{\alpha}$-semistable} if for any rank $1$ meromorphic Higgs subbundle $(\F , \theta|_{\F})$ we have
 \begin{gather*}
 \pardeg (\F ) \leq \frac{\pardeg (\E )}2.
 \end{gather*}
\end{defn}

Let $M^{\vec{\alpha}-s}$ and $M^{\vec{\alpha}-ss}$ be the functors
\begin{gather*}
 \mbox{Schemes} \to \mbox{Sets}
\end{gather*}
associating to a scheme $S$ the set of $S$-equivalence classes of holomorphic vector-bundles
\begin{gather*}
 \E \to \CP1 \times S
\end{gather*}
endowed with a morphism
\begin{gather*}
 \theta\colon \ \E \to \E \otimes p_1^* K_{\CP1}
\end{gather*}
(where $p_1\colon \CP1 \times S \to \CP1$ stands for the first projection) and with filtrations
\begin{gather*}
 \ell_j \subset \E|_{\{ t_j\} \times S},
\end{gather*}
where $\ell_j$ is a line-subbundle preserved by $\res_{\{ t_j\}
 \times S} (\theta )$, and such that for any geometric point $s \in
 S$ the triple $\big(\E , \theta, \{ \ell_j \}_{j=1}^4\big)$ defines an
 $\vec{\alpha}$-stable (respectively, $\vec{\alpha}$-semistable)
 meromorphic Higgs bundle of rank $2$ over $\CP1$ with a compatible
 parabolic structure {and local form of the Higgs field specified
 in Section~\ref{ssec:param_geo}; for simplicity of the notation,
 we do not include the parameters $a,b,\ldots $ in the notation of the functors,
 but we tacitly fix them.}

\begin{thm}[\cite{nit}] There exist quasi-projective coarse moduli schemes
 \begin{gather*}
 \Mod^{\vec{\alpha}-s} \big(\CP1, D, 2\big) \subseteq \Mod^{\vec{\alpha}-ss} \big(\CP1, D, 2\big)
 \end{gather*} for the functors $M^{\vec{\alpha}-s}$ and $M^{\vec{\alpha}-ss}$, respectively.
\end{thm}

The Hitchin map introduced in Definition \ref{def:Hitchin_map} induces morphisms
\begin{gather*}
 h\colon \ {\overline{\Mod^{\vec{\alpha}-s} \big(\CP1, D, 2,d \big)}} \to \mathcal{B},
\end{gather*}
and for any $(F,G) \in \mathcal{B}$ we denote by
\begin{gather*}
 {\overline{\Mod^{\vec{\alpha}-s} \big(\CP1, D, 2,d \big)}}_{(F,G)}
\end{gather*}
its fiber over $(F,G)$.

In \cite{ISS, ISS2} we studied stability of Higgs bundles with spectral curves such that the curve in the associated elliptic fibration is of type $I_n$ $(1 \leq n \leq 3)$, $II$, $III$ or $IV$. The only new case that appears in the cases treated here is that of a spectral curve $\Sigma$ such that the curve in the associated elliptic fibration is of type $I_0^* = \tilde{D}_4$. As we will see in Lemma~\ref{lem:I0Fib}, this case only appears if the spectral curve $\Sigma$ is a section of $p\colon {\mathbb {F}}_2 \to \CP1$ taken with multiplicity two. From now on, we assume that $\Sigma$ is of this type, and in addition that
\begin{gather*}%\label{eq:pardegE0}
 \pardeg (\E ) = 0.
\end{gather*}
Up to tensoring with a rank $1$ meromorphic Higgs bundle, the curve $\Sigma = \Sigma_{(0,0)}$ is the image of the $0$-section $F = 0$, $G = 0$.

 \begin{assn} From now on we only consider Higgs bundles satisfying
 \begin{gather*}
 \pardeg (\E ) = 0.
 \end{gather*}
 \end{assn}

This assumption implies that
\begin{gather*}
 \sum_{j=1}^4 (\alpha_j^+ + \alpha_j^-) \in \Z,
\end{gather*}
so the space of possible parabolic weights belong to a finite union of affine hyperplanes in $[0,1)^8$.
We will say that the parabolic weights are {\it generic} if for any choice of $\varepsilon_j \in \{ +, - \}$ we have
\begin{gather*}
 \sum_{j=1}^4 \alpha_j^{\varepsilon_j} \notin \Z.
\end{gather*}
Generic parabolic weights belong to the complement of finitely many proper affine subspaces of the affine hyperplanes of all possible parabolic weights.
For generic parabolic weights, $\vec{\alpha}$-semi-stability is equivalent to $\vec{\alpha}$-stability, because the parabolic degree of an invariant subobject
$\F$ is never an integer.

\begin{prop}\label{prop:Hitchin_fiber}
For generic parabolic weights the Hitchin fiber
\begin{gather*}
 \overline{\Mod_t^{\vec{\alpha}-s} \big(\CP1, D, 2, d\big)}_{(0,0)}
\end{gather*}
over $(0,0) \in \mathcal{B}$ is of type $I_0^*$.
\end{prop}

\begin{proof}It follows from~\cite[Theorem~0.2]{BB} that the moduli space in the assertion is a smooth complete hyperK\"ahler manifold.
Moreover, according to~\cite[Lemma~1]{Sz_PW}, the Hitchin map defined on the moduli space is proper.
It is well-known that the generic Hitchin fibers are elliptic curves; indeed, they are (up to a degree twist) the Jacobians of the elliptic curves
appearing in the pencil of spectral curves to be studied in Example~\ref{exam:F2}.
Now, the elliptic fibration of Example~\ref{exam:F2} admits a central fiber of type $I_0^*$, therefore the monodromy matrix
of the pencil of spectral curves is conjugate to
\begin{gather*}
 \begin{pmatrix}
 -1 & 0 \\
 0 & -1
 \end{pmatrix}.
\end{gather*}
It \looseness=-1 follows from this that the monodromy matrix of its dual fibration is also conjugate to the same matrix. As the monodromy matrix uniquely determines the Kodaira type of a~singular fiber in an elliptic surface, we deduce that the only singular fiber of the Hitchin map is of type~$I_0^*$ too.

Alternatively, it follows from~\cite{CB} that an open subset of the moduli space is a smooth quiver variety associated to the quiver~$D_4^{(1)}$ and
dimension vector equal to the positive minimal integer generator of the radical of the associated quadratic form.
This quiver variety is isomorphic to a~minimal resolution of the surface singularity of type~$D_4$.
The exceptional divisors of this resolution lie in the singular Hitchin fiber, so the singular Hitchin fiber contains a curve of type~$D_4$;
the only Kodaira fiber containing a curve of this type is~$I_0^*$.
\end{proof}

\section{Singular fibers in elliptic fibrations} \label{sec:elliptic}

An elliptic fibration is a map $\pi \colon X\to C$ where $X$ is a
compact complex surface, $C$ is a compact complex curve and the
generic fiber is an elliptic curve, i.e., smoothly diffeomorphic to
the 2-torus. In the following we will consider only those elliptic
surfaces which admit a section, i.e., \emph{basic elliptic surfaces} in
the terminology of \cite[Section~3.1.4]{FrMr}. This restriction is
justified by the existence of the Hitchin section~\cite{Hit_Teich}
in moduli spaces of Higgs bundles, whose construction can easily be
generalized to moduli spaces of (logarithmic or irregular) singular Higgs bundles.
Alternatively, it follows from \cite[Section~4]{Sz-spectral} that
the elliptic fibrations we will encounter all originate from a pencil
by blow-ups, hence all admit sections. The singular fibers in an
elliptic fibration have been classified by Kodaira in \cite{Kodaira}.
From Kodaira's list of singular fibers we will only need the fibers
$\Et8$, $\Et7$, $\Et6$, $II$, $III$, $IV$ and $I_n^*$ $(n\geq 0)$, see
Figs.~\ref{fig:regi} and~\ref{fig:34} for their description. (The
fiber of type $II$ is not given pictorially -- it is the cusp fiber,
which is topologically a sphere with a singular point, where the
singularity can be modeled by a cone on the trefoil knot. In $\CP{2}$
such a curve can be given in homogeneous coordinates $[x : y : z]$ by the equation $zy^2=x^3$.)

\begin{figure}[t]\centering
\includegraphics[width=10cm]{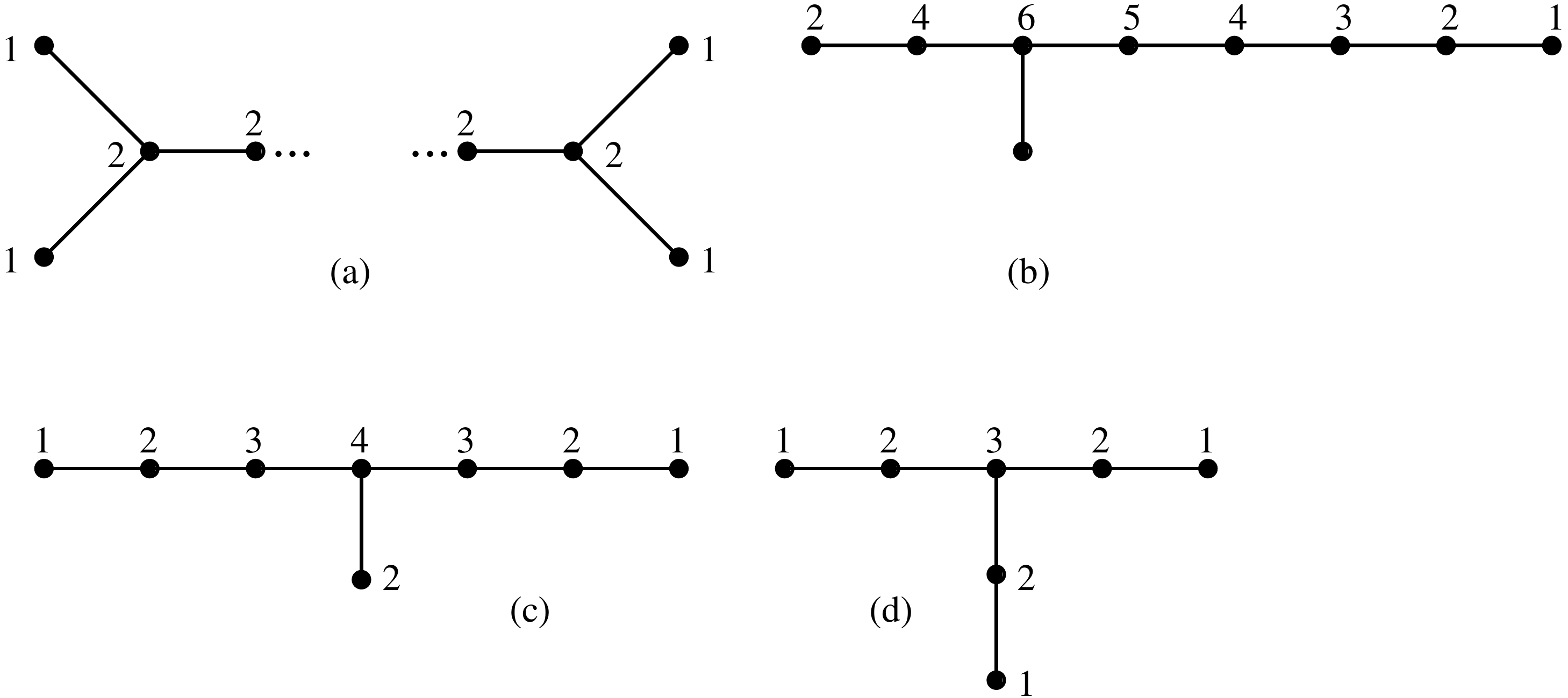}
\caption{Plumbings of singular fibers
of types (a) $I_n^*$, (b) ${\tilde {E}}_8$, (c) ${\tilde {E}}_7$, and
(d) ${\tilde {E}}_6$. Integers next to vertices indicate the
multiplicities of the corresponding homology classes in the fiber. All
dots correspond to rational curves with self-intersection $-2$. In
$I_n^*$ we have a total of $n+5$ vertices; in particular, $I_0^*$ admits
a vertex of valency four.}\label{fig:regi}
\end{figure}

\begin{figure}[t]\centering
\includegraphics[width=4.5cm]{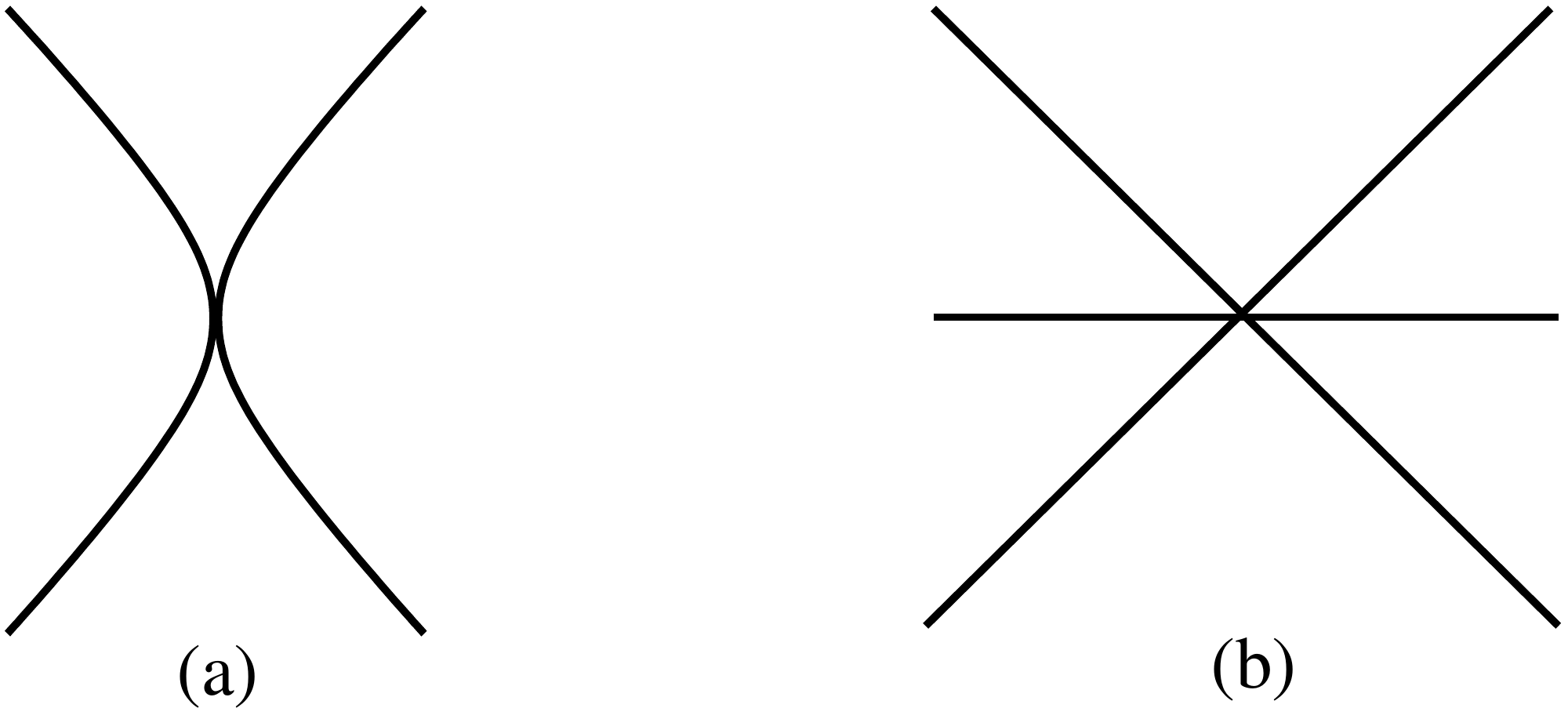}
\caption{Singular fibers of types $III$ and $IV$ in elliptic
 fibrations. In (a) the two curves are tangent with multiplicity
 two, and in (b) the three curves pass through one point and
 intersect each other there transversely. {All curves have
 multiplicity one and self-intersection $-2$.}}\label{fig:34}
\end{figure}

Various possibilities of singular fibers in an elliptic fibration on
a rational elliptic surface were analyzed in detail in \cite{Miranda, Persson, SSS}. In particular, we have

\begin{thm}[{\cite[Section~6]{SSS}}] Suppose that $\pi \colon X\to \CP{1}$ is an elliptic fibration on a
rational elliptic surface $X$. Assume furthermore that $\pi$ admits a
singular fiber of type $I_n^*$. Then $n\leq 4$. Furthermore, if $\pi
$ has exactly two singular fibers, and one of them is of type $I_n^*$,
then $n=0$ and both fibers are of type $I_0^*$.
\end{thm}

\begin{cor}\label{cor:ins}There is no elliptic fibration on the rational elliptic surface
with exactly two singular fibers, one of which is $I_n^*$ with $n>0$.
\end{cor}

We analyze the fibration with two $I_0^*$ fibers a little further.
We can construct such a fibration by considering the following
pencils.
\begin{Example}\label{exam:F2}
Consider the Hirzebruch surface ${\mathbb {F}}_2$ with infinity
section of square $(-2)$ and section of square $2$. Let $C_0$ be the
curve given by a section with multiplicity 2. Consider $C_{\infty}$ as
the union of the section at infinity with multiplicity 2 together
with 4 distinct fibers of the ruling on ${\mathbb {F}}_2$. It is easy
to see that the two curves are homologous, and indeed, can be given as
zero-sets of two sections of the same holomorphic line bundle, hence
there is a pencil of curves containing both. Blowing up the four
basepoints (which are the intersections of the fibers with the
section) twice, we get an elliptic fibration on the 8-fold blow-up of
${\mathbb {F}}_2$ (which is a rational elliptic surface), with two
singular fibers, each of type $I_0^*$. For the blow-up process, see
Fig.~\ref{fig:pot1}.
\end{Example}

\begin{figure}[t]\centering
\includegraphics{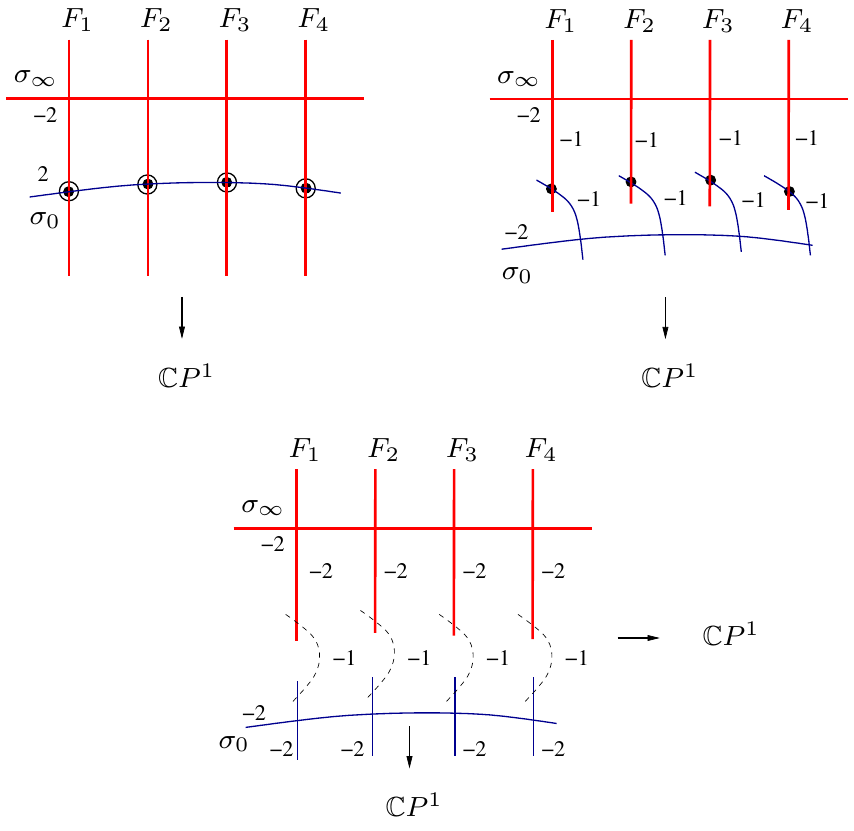}
\caption{On the left the Hirzebruch surface (with the ruling), together with the
curves $C_{\infty}=\sigma _{\infty}\cup F_1\cup F_2\cup F_3\cup F_4$
and $C_0=\sigma _0$ is depicted. The concentric circles indicate double
blow-ups -- after one blow-up at each of the four points we get the
middle configuration, while four further blow-ups provide
the right hand diagram, with two $I_0^*$ fibers.}\label{fig:pot1}
\end{figure}

In some sense the converse of the above example also holds.
Indeed, if we get a fibration by blowing up a pencil on
${\mathbb {F}}_2$ containing $C_{\infty}$, then the pencil has at least
4 basepoints, each on different fibers of the ruling on ${\mathbb {F}}_2$,
hence the fibration has at least four disjoint
$(-1)$-sections. In addition, these sections intersect different curves
in the fiber (of type~$I_0^*$) coming from the curve $C_{\infty}$ at
infinity.

\begin{lem}\label{lem:I0Fib}
Suppose that $\pi \colon X\to \CP{1}$ is an elliptic fibration with
two singular fibers $F_0$, $F_1$, both of type $I_0^*$, and the
fibration admits four disjoint sections $($of homological square~$(-1))$ which intersect different leaves of~$F_0$. Then this fibration
comes from the construction of Example~{\rm \ref{exam:F2}}.
\end{lem}
\begin{proof}
Let $G_i^1$, $G_i^2$, $G_i^3$ and $G_i^4$ denote the $(-2)$-curves in the
fiber $F_i$ (with $i=0,1$) correspon\-ding to the leaves of the plumbing
presentation of these fibers of type~$I_0^*$. Suppose furthermore
that the four sections are denoted by $E_1$, $E_2$, $E_3$ and $E_4$, with
the understanding that $E_j$ intersects $G_0^j$ (and no other
$G_0^k$), see Fig.~\ref{fig:pot3}.

\begin{figure}[t]\centering
\includegraphics{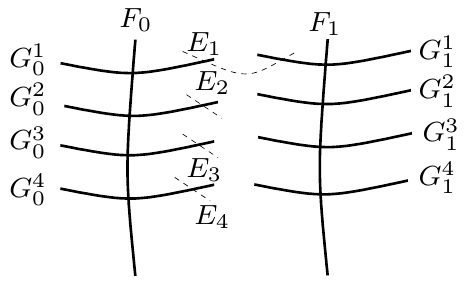}
\caption{The two $I_0^*$ fibers and the sections $E_i$ are
shown in the diagram. $E_1$ intersects $G_1^1$; the further intersections
of the $E_i$'s with $i>1$ provide various cases, analysed in the proof.}\label{fig:pot3}
\end{figure}

We can also assume (after possibly renumbering) that $G_1^1$
is the curve among the $G_1^j$'s intersected by the most~$E_i$'s. Note
that each section intersects a fiber of type $I_0^*$ only in curves
corresponding to leaves (since the central curve has multiplicity 2,
while a section intersects the fiber exactly once (with
multiplicity)).

If $G_1^1$ is intersected by a single $E_i$, then by our choice
each $G_1^j$ is intersected by a single section, hence blowing down
first the sections and then the images of the curves $G_1^j$'s,
we get the configuration of curves described in Example~\ref{exam:F2}.

Suppose now that $G_1^1$ is intersected by two $E_i$'s, say by $E_1$
and $E_2$. Blow these curves and~$G_0^1$,~$G_0^2$ down. In the result
the curve $G_1^1$ will be a $(+2)$-curve, hence the complement of this
curve should be negative definite. On the other hand, $E_3$
(intersecting $G_0^3$) also intersects the fiber~$F_1$, in the curve
$G_1^2$ (after possible renumbering). Now blowing down $E_3$ and
$G_1^2$, the curve $G_0^3$ will be a complex curve of
self-intersection 0. Since a complex curve does not vanish in
homology, it defines a nontrivial homology class in the complement of
the $(+2)$-curve with zero self-intersection, providing a~contradiction.

Assume now that $G_1^1$ is intersected by three sections, say $E_1$, $E_2$ and $E_3$. Blowing these and the corresponding curves from $F_0$
(i.e., $G_0^1$, $G_0^2$, $G_0^3$) down, the curve defined by~$G_1^1$ will have self-intersection $(+4)$ in $\CP{2}\# 3\CPbar$ with four curves $\big(G_0^4, G_1^2, G_1^3, G_1^4\big)$ in its complement, linearly independent in homology. This is a~contradiction again.

Assume finally that all four sections intersect $G_1^1$. Blowing the sections and then the curves~$G_0^i$ down, we get the image of $G_1^1$ in $\CP{2}\# \CPbar$ of self-intersection $(+6)$, with three curves in the complement, defining linearly independent homology classes, providing the last desired contradiction.
\end{proof}

\begin{Example}\label{exam:CP2I0}
In a similar manner, we can specify
a pencil of curves in the complex projective plane which (after 9 blow-ups)
provides an elliptic fibration with two singular fibers, both of type~$I_0^*$. Indeed, let $S_{\infty}=\cup _{i=1}^3\ell _i$ be the union of three lines
$\ell _1$, $\ell _2$, $\ell _3$ passing through
a~fixed point $P\in \CP{2}$, each line with multiplicity one.
Define $S_0$ to be the union $L_1\cup L_2$, where $L_1$ passes through
$P$ while $L_2$ does not; we consider~$L_1$ with multiplicity~1 and~$L_2$ with multiplicity~2, see Fig.~\ref{fig:pot2}.
\begin{figure}[t] \centering
\includegraphics[scale=0.8]{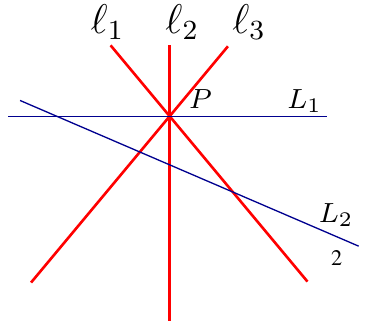}
\caption{Curve configuration in ${\mathbb {C}}P^2$ used in Example~\ref{exam:CP2I0}. }\label{fig:pot2}
\end{figure}
In this way both $S_{\infty}$ and $S_0$ are (singular) curves of degree 3, and the resulting pencil (with four basepoints,
the three intersection points $\ell _i\cap L_2$ and $P$) will define an elliptic fibration with exactly two $I_0^*$ fibers and no further singular fibers.
\end{Example}

\begin{Example}
A more symmetric presentation of the same fibration can be given as
follows. Consider $\CP{1}\times \CP{1}$ (which is the Hirzebruch
surface ${\mathbb {F}}_0$), together with the ruling $r\colon
\CP{1}\times \CP{1}\allowbreak \to \CP{1}$. Choose two sections $S_0$ and
$S_{\infty}$ of the ruling, and four distinct fibers $F_1$, $F_2$, $F_3$,~$F_4$.
Define the two curves $C_0$, $C_{\infty}$ by
\[
C_0=S_0\cup F_1\cup F_2, \qquad C_{\infty}=S_{\infty }\cup F_3\cup F_4,
\]
where each $F_i$ appears with multiplicity 1, while both $S_0$ and $S_{\infty}$
appear with multiplicity 2, see Fig.~\ref{fig:pot4}.
\begin{figure}[t] \centering
\includegraphics{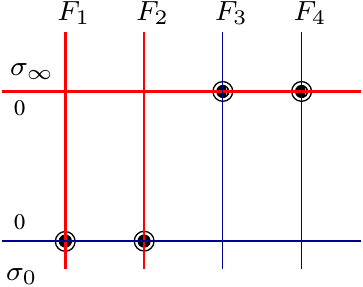}
\caption{The symmetric curve configuration in ${\mathbb {F}}_0={\mathbb {C}}P^1\times {\mathbb {C}}P^1$ providing a elliptic fibration with two $I_0^*$ fibers.}\label{fig:pot4}
\end{figure}
The two curves obviously generate a pencil on $\CP{1}\times \CP{1}$ with
four basepoints. Blowing up each basepoints twice, we get the desired
elliptic fibration with two singular fibers, each of the type of $I_0^*$,
originated from the curves $C_0$ and $C_{\infty}$.
\end{Example}

\section{The proofs}\label{sec:proofs}
In this section we give the proofs of
Theorems~\ref{thm:OnePoleUntwisted}--\ref{thm:FourPoles}. % described in Section~\ref{sec:intro}.
For sake of completeness and in order to have a uniform treatment, we include and provide references to our previously obtained results whenever applicable.

\subsection{The case of one pole}
\begin{proof}[Proof of Theorems~\ref{thm:OnePoleUntwisted} and \ref{thm:OnePoleTwisted}]
These theorems are special cases of \cite[Theorem~1.1]{ISS} and
\cite[Theorem~1.3]{ISS}, respectively; hence we do not give any
further details here.
\end{proof}

\subsection{The case of two poles}
This case has been extensively studied in \cite{ISS2}. Recall that in this
case we have two subcases depending on the multiplicities of the poles:
these multiplicities can be~(a)~2 and~2 or (b)~3 and~1.

If the two multiplicities are 2 and 2, then the fiber at infinity is
either of type $I_2^*$, $I_3^*$ or $I_4^*$. According to
Corollary~\ref{cor:ins}, elliptic fibrations with those singular fibers
on a rational elliptic surface have at least two further singular fibers,
hence they do not occur in the cases under the current investigation.
(For the complete picture in this case see
\cite[Theorems~{2.1,~2.2,~2.3}]{ISS2}, depending on whether the polar
parts are twisted or untwisted. See also \cite[Table~1]{ISS2}.)

\begin{proof}[Proof of Theorems~\ref{thm:TwoPolesUT} and~\ref{thm:TwoPolesTT}]
Suppose now that the two mul\-ti\-pli\-ci\-ties are 3 and 1. By
examining \cite[Theorems~{2.4,~2.5,~2.6,~2.7}]{ISS2}, we see that there
are two cases with single singular fibers in the Hitchin fibration,
as given in the two theorems.
\end{proof}

\subsection{The case of three poles}
Suppose now that there are three poles. In this case the fiber at
infinity in the elliptic fibration is either of type $I_1^*$ (when the
polar part at the pole with multiplicity 2 is untwisted (i.e., regular
semi-simple)) or of type $I_2^*$ (when this polar part is twisted, i.e., has non-vanishing nilpotent part).

\begin{proof}[Proof of Theorem~\ref{thm:ThreePoles}]
According to Corollary~\ref{cor:ins}, next to these singular fibers
we cannot have a single singular fiber in any elliptic fibration on a
rational elliptic surface, verifying the statement of the theorem.
\end{proof}

\subsection{The case of four poles}
\begin{proof}[Proof of Theorem~\ref{thm:FourPoles}]
Consider the curve $C_{\infty}\subset {\mathbb {F}}_2$ which is the union
of the section at infinity (with self-intersection $(-2)$) with multiplicity
2 and four fibers of the ruling on ${\mathbb {F}}_2$. We can choose
coordinates on $\CP{1}$ in such a way that the four fibers are
over $0$, $1$, $\infty$ and~$t$, where $t\in \C$ described the holomorphic type
of the fiber at infinity.

By Lemma~\ref{lem:I0Fib} the pencil containing the above curve
provides an elliptic fibration with two~$I_0^*$ fibers if and only if
the pencil contains a section of the ruling $p\colon {\mathbb
 {F}}_2\to \CP{1}$ (which is of multiplicity~2).
 {By virtue of Proposition~\ref{prop:Hitchin_fiber}, in this case
 the Hitchin fibration on the extended moduli space $\overline{\Mod}$
 of~\eqref{eq:compactification} also has a singular fiber of type $I_0^*$
 for generic parabolic weights.}
 In particular, the
pencil has four basepoints, one on each of the four distinguished fibers
of the ruling of ${\mathbb {F}}_2$. Suppose that the four basepoints
are at $a$, $b$, $c$, $d$ in the fibers over the points $0,1,t,\infty \in
\CP{1}$.

The ruling $p\colon {\mathbb {F}}_2\to \CP{1}$ admits a complex
3-dimensional family of sections (all disjoint from the section at
infinity), which can be given by homogeneous degree-2 polynomials on
$\CP{1}$. In particular, if the homogeneous coordinates on $\CP{1}$
are denoted by $[u:v]$, then a section is of the form $\alpha
v^2+\beta uv+\gamma u^2$. The section intersects the fiber over $0$
(given by the homogeneous coordinates $[0:1]$) at $\alpha$,
hence we need $\alpha =a$. Similarly, over $\infty=[1:0]$
the intersection point is $\gamma$, so we have $\gamma =d$.
Finally, over the point $1=[1:1]\in \CP{1}$ we get that
$\alpha +\beta +\gamma =b$, implying $\beta =b-a-d$.
In conclusion, the unique section intersecting the
fibers over $0$, $1$, $\infty$ in the prescribed points $a$, $b$, $d$
can be given by the equation
\[
av^2+(b-a-d)uv+du^2.
\]

This section takes the value
$a+(b-a-d)t+dt^2$ at $t=[t:1]$, providing the condition
\[
c= a+(b-a-d)t+dt^2
\]
for the existence of the desired section (and therefore for the other
fiber to be of type $I_0^*$), and concluding the proof of the statement.
\end{proof}
{This argument concludes the proof of Theorems~\ref{thm:OnePoleUntwisted}--\ref{thm:FourPoles}. We are left
 with Theorem~\ref{thm:I0*osszeseset}, the proof of which will be
 provided in the next section.}

\section[Pencils and fibrations with fiber $I_0^*$]{Pencils and fibrations with fiber $\boldsymbol{I_0^*}$}\label{sec:app}

We devote this section to the proof of
Theorem~\ref{thm:I0*osszeseset}. We will describe pencils in the
Hirzebruch surface ${\mathbb {F}}_2$, and sometimes (for the sake of
convenience) we also provide a pencil in the complex projective plane
giving rise to a fibration with the same desired properties. There is a
rather simple procedure for converting pencils in ${\mathbb {F}}_2$ to
pencils in $\CP{2}$ (by blowing up and down); in most cases we will
omit the actual steps of this procedure.

The possible combinations of singular fibers in elliptic fibrations on
a rational elliptic surface have been determined in \cite{Miranda, Persson},
see also \cite{SSS}. Recall the types of elliptic singular
fibers listed in Section~\ref{sec:elliptic}; there are a few further
types we need to encounter in the following. $I_1$ denotes the nodal
fiber (also called a fishtail fiber), which is a nodal elliptic curve,
and topologically an immersed 2-sphere with one positive transverse
double point. The fiber~$I_n$ consists of $n$ rational curves~$E_0,
\ldots, E_{n-1}$ such that $E_i$ transversely intersects~$E_{i-1}$ and
$E_{i+1}$ in two distinct points (one each), and the indices are viewed mod
$n$. In particular, $I_2$ consists of two rational curves~$E_0$ and~$E_1$ intersecting each other in two distinct points~-- as the two points
converge to each other, the limit of such a sequence of singular
fibers is a type~$III$ fiber.

There are 19 distinct configurations of singular fibers
on $\CP{2}\# 9 {\overline {\CP{2}}}$ containing $I_0^*$ as one of
the singular fiber, which
we list now (separating the singular fibers in a fibration with an
addition sign), grouped in four groups. Below we list the fibers
next to the (always existing) $I_0^*$ fiber:

\begin{table}[h!]\centering
\begin{tabular}{llllll}
$I_0^*$,& $I_4+2I_1$,& $IV+II$,& $IV+2I_1$, & $I_3+II+I_1$,& $I_3+3I_1$; \\
$2III$,& $III+I_2+I_1$,& $III+II+I_1$,& $III+3I_1$; & &\\
$3I_2$,& $2I_2+2I_1$,& $I_2+2II$,& $I_2 +II +2I_1$,& $I_2+4I_1$; & \\
$3II$,& $2II+2I_1$,& $II+4I_1$,& $6I_1$. & &
\end{tabular}
\caption{Possible singular fiber combinations next to $I_0^*$.}\label{tab:singfibers}
\end{table}

In the following we give examples of pencils of curves in ${\mathbb
 {F}}_2$ (and sometimes in $\CP{2}$) providing elliptic fibrations on rational
elliptic surfaces containing the configurations listed above.
The curves are chosen so that we do have control on the further
singular fibers of the fibration. In the following we will describe
pencils in ${\mathbb {F}}_2$ and in $\CP{2}$ by specifying two curves
(in each complex surface) and consider the pencil generated by them~-- and the singular fibers will then be in the fibrations we get by
blowing up {the base points of} these pencils.

The curve $C_{\infty}\subset
{\mathbb {F}}_2$ will be part of all the pencils; $C_{\infty}$
consists of the union of four fibers $F_1$, $F_2$, $F_3$, $F_4$ of the ruling
${\mathbb {F}}_2\to \CP{1}$ (all with multiplicity~1) and the section
at infinity of the ruling (this curve with multiplicity~2)~-- as it
was already described in Example~\ref{exam:F2}.
Similarly, one curve in each pencil in~$\CP{2}$ will be the same:
$S_{\infty}=\ell _1\cup \ell _2\cup \ell _3$ is the cubic curve we get
by taking the union of three projective lines $\ell _i$ ($i=1,2,3$)
(all with multiplicity~1) passing through a fixed point $P\in
\CP{2}$. The curve $S_{\infty}$ is shown by the thick red lines on the right of
the figures below. (This curve already appeared in Example~\ref{exam:CP2I0}.)

We say that a pencil in ${\mathbb {F}}_2$ \emph{has a section}
if one of the curves in the pencil other than $C_{\infty}$
has more than one components~-- in this case each component
of this curve is indeed a~section of the ruling of
the Hirzebruch surface~${\mathbb {F}}_2$.

In some cases we need to see more than two singular fibers in the
fibration. To this end, we will apply a principle (formulated in
Lemma~\ref{lem:princ} below) and a method of generating more curves in
a pencil (described after Lemma~\ref{lem:princ}).

Suppose that the pencil is given by the two curves $C_0$ and
$C_{\infty}$ in ${\mathbb {F}}_2$ (or $S_0$ and $S_{\infty}$ in
$\CP{2}$), intersecting each other in points $P_1, \ldots , P_k$ with
multiplicities $n_1, \ldots ,n_k$, satisfying $\sum_{i=1}^kn_i=8$ in
${\mathbb {F}}_2$ and $\sum_{i=1}^kn_i=9$ in $\CP{2}$. (In case
$n_i>1$, not only $P_i$ is fixed, but also the higher order tangencies
are part of the given data, for example by specifying the points in
the exceptional divisor of the blow-up the strict transforms of the
curves pass through.) We mentioned that the elements of matrices in
equations~\eqref{eq:4poles_ut_u} and~\eqref{eq:4poles_t_u} encode the
base locus of the pencil, hence $P_1, \ldots , P_k$ are determined by
these data. Next, we will construct some examples where we choose
these parameters $P_1, \ldots , P_k$ appropriately on four
distinguished fibers of the ruling, so that all combinations of
further singular fibers appear. It is easy to see that these chosen
parameters can be transformed into $a_{\pm}, \dots, d_{\pm}, a, \dots,
d$. (For convenience, from now on we will not insist on having the
distinguished fibers over the points~$0$, $1$, $\infty$~-- the application
of a simple M\"obius transformation would provide this extra feature.)

Suppose now that $C$ (or $S$) is a complex curve in ${\mathbb {F}}_2$
(or in $\CP{2}$, respectively) which is homologous to $C_0$ and
$C_{\infty}$ (or $S_0$ and $S_{\infty}$) and passes through~$P_1,
\ldots , P_k$, intersecting $C_0$ (or~$S_0$) in those points with
multiplicities $n_1, \ldots ,n_k$ (and with the required higher order
tangencies in case $n_i>1$).

\begin{lem} \label{lem:princ} Under the above circumstances $C$ $($or~$S)$ is in the pencil in ${\mathbb {F}}_2$ generated by $C_0$ and $C_{\infty}$ $($or $S_0$ and $S_{\infty}$ in case we work in $\CP{2})$.
\end{lem}
\begin{proof}
Take any point $P\in C\setminus \{ P_1, \ldots , P_k \}$. There exists a unique $t\in \CP1$ such that~$C_t$ passes through~$P$. We then see that~$C$ intersects~$C_t$ with multiplicity $>\sum_{i=1}^kn_i$, which is by assumption the self-intersection number of~$C$. It then follows that $C = C_t$.
\end{proof}

We can construct curves in ${\mathbb {F}}_2$ disjoint from the section
$\sigma _{\infty}$ at infinity and intersecting the fiber of the ruling
twice (possibly once, with multiplicity two) in the following way.
Since the curve is disjoint from $\sigma _{\infty}$, it is a double section
of the bundle ${\mathcal {O}}(2)\to \CP{1}$. Such double sections
can be given by sections of ${\mathcal {O}}(4)\to \CP{1}$:
for a given section $\sigma$ of ${\mathcal {O}}(4)$, use
the identification ${\mathcal {O}}(2)\otimes {\mathcal {O}}(2)\cong
{\mathcal {O}}(4)$ and get a double section of ${\mathcal {O}}(2)$:
over $P\in \CP{1}$ take those points $\zeta \in {\mathcal {O}}(2)$ which
satisfy $\zeta \otimes \zeta =\sigma _P$, see~\eqref{eq:char-poly1} with
the notational change of denoting
$G_{\theta}$ by $\sigma$.
There are two such points ($\zeta $ and $-\zeta $)
if $\sigma _P\neq 0$ and there is a single one if $\sigma _P=0$.
(Indeed, this is the classical construction of double branched covers.)
In turn, sections of ${\mathcal {O}}(4)\to \CP{1}$ can be given by
homogeneous degree-4 polynomials in the homogeneous coordinates
$[u:v]$ on $\CP{1}$.

Here are some instructive examples of this phenomenon:
\begin{itemize}\itemsep=0pt
\item The polynomial $u^4$ (as a section of ${\mathcal {O}}(4)$)
provides two sections in ${\mathcal {O}}(2)$ which are tangent to each
other over the point $[0:1]\in \CP{1}$.
\item The polynomial $u^2v^2$ gives rise to two sections of ${\mathcal
 {O}}(2)$, intersecting each other transversally in two points over $[0:1]$
and $[1:0]\in \CP{1}$.
\item The polynomial $u^2v(u+v)$ defines a double section in
 ${\mathcal {O}}(2)$ which has a node over $[0:1]\in \CP{1}$,
and is tangent to the fibers of the ruling at $[1:0]$ and $[-1:1]\in \CP{1}$.
\item The polynomial $u^3v$ will define a cuspidal curve which is a
double section in ${\mathbb {F}}_2$ -- the cusp point is over $[0:1]$
and the fiber over $[1:0]$ is tangent to the cuspidal curve.
\end{itemize}

We will start our examples with configurations admitting sections.

\begin{Example}[fibration with another $I_0^*$]\label{exam:I0I0}
If $C_0$ is a section
of the ruling ${\mathbb {F}}_2\to \CP{1}$ with square 2 and we take it
with multiplicity 2, the pencil generated by $(C_0, C_{\infty})$
provides a fibration with two singular fibers, both of type $I_0^*$, cf.\ Fig.~\ref{fig:I0}.
\begin{figure}[t] \centering
\includegraphics[scale=0.8]{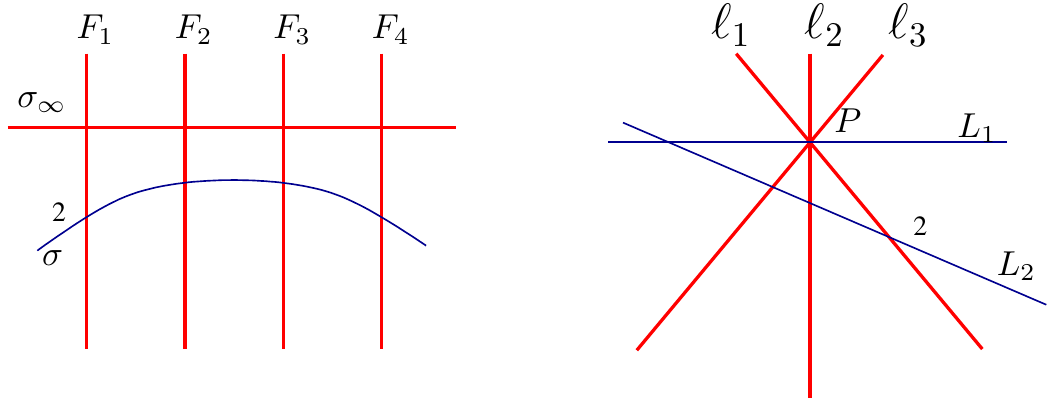}
\caption{The diagram on the left shows two curves in ${\mathbb
 {F}}_2$, one of them is the curve $C_{\infty}$, while the other
 one is a section of the ruling with multiplicity two (denoted by~$\sigma$). Thick (red) lines provide one of the curves, while the
 thin (blue) segment defines the other curve.
The fibration defined by the pencil contains two
 $I_0^*$ fibers. The two degree-3 curves in $\CP{2}$ on the right
 generate a pencil which gives rise to a fibration admitting two~$I_0^*$ fibers.}\label{fig:I0}
\end{figure}
In a similar fashion, by taking two lines $L_1, L_2\subset \CP{2}$ such
that $L_1$ passes through the fixed point $P$ while $L_2$ avoids it,
and defining $S_0$ as the union of $L_1$ and $L_2$ (with $L_1$ of
multiplicity 1 and $L_2$ of multiplicity 2) we get a pencil on
$\CP{2}$ induced by $(S_0, S_{\infty})$ so that the resulting elliptic
fibration has two $I_0^*$ fibers (cf.\ Example~\ref{exam:CP2I0} and Fig.~\ref{fig:I0}).
\end{Example}

\begin{Example}[fibration with an $I_4$ fiber]\label{exam:I0I4}
A fibration with an $I_4$ fiber (next to the $I_0^*$ coming from
$C_{\infty}$) can be given in ${\mathbb {F}}_2$ by considering two
sections $\sigma _1$, $\sigma _2$ of the ruling (each with multiplicity
1) intersecting each other once on $F_1$ and once on $F_4$, as shown
by the left diagram of Fig.~\ref{fig:I4}. A similar pencil in
$\CP{2}$ can be given by taking $S_0=L_1\cup L_2\cup L_3$, where $L_i$
are three distinct lines, $L_1$ passes through $P$, while $L_2$, $L_3$
do not, and moreover the intersection point $L_2\cap L_3$ is on one of
the~$\ell _i$; the configuration is shown by the right diagram of
Fig.~\ref{fig:I4}. Note that by the classification of possible
singular fibers, it follows that next to the $I_0^*$ and $I_4$ fibers
there is a single possibility for other singular fibers: these must be
two fishtail ($I_1$) fibers.
\begin{figure}[t] \centering
\includegraphics[scale=0.8]{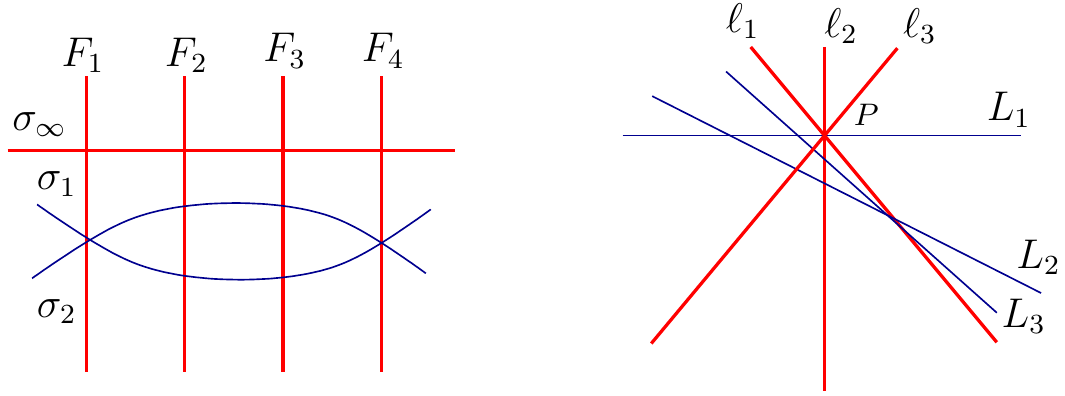}
\caption{The diagram on the left shows two sections in ${\mathbb
 {F}}_2$ intersecting each other in two points of~$C_{\infty}$. (The sections are denoted by $\sigma _1$ and $\sigma
 _2$.) Once again, thick (red) lines provide one of the curves, while
 the thin (blue) segments define the other curve. The pencil
 generated by these curves provides a fibration with an~$I_0^*$ fiber
 and an~$I_4$ fiber. The two degree-3 curves in $\CP{2}$ on the
 right diagram generate a pencil which gives rise to a fibration
 admitting an~$I_0^*$ fiber and an~$I_4$ fiber.}\label{fig:I4}
\end{figure}
\end{Example}

\begin{Example}[fibration with a type $IV$ fiber]
To get a fibration with a fiber of type $IV$ next to $I_0^*$, take
two sections in ${\mathbb {F}}_2$, which are now tangent to each other at
a point of $F_1$, see the left diagram of Fig.~\ref{fig:IV}. In
$\CP{2}$ consider $S_0$ to be three distinct lines $L_1$, $L_2$, $L_3$
again, so that $L_1$ passes through $P$, while all three lines pass
through $P'\in \CP{2}$ which is distinct from~$P$, see the right
diagram of Fig.~\ref{fig:IV}.
\begin{figure}[t] \centering
\includegraphics[scale=0.8]{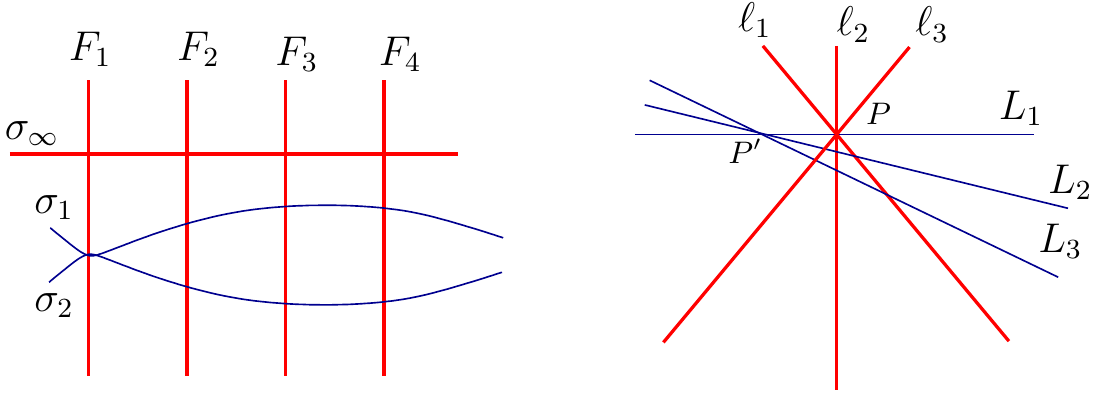}
\caption{The diagram depicts two pencils (one in ${\mathbb {F}}_2$
 and the other one in $\CP{2}$) giving fibrations with an
 $I_0^*$ fiber and a type $IV$ fiber. The conventions are the same as
 earlier.}\label{fig:IV}
\end{figure}

In order to get a pencil generating a fibration with an $I_0^*$, a
type $IV$ and a type $II$ fiber, consider the sections $u^4$ and
$uv^3$ in ${\mathcal {O}}(4)$, giving rise to two double sections in
${\mathbb {F}}_2$. Let $F_1$, $F_2$, $F_3$, $F_4$ be chosen as follows:
$F_1$ is over $[0:1]$ (where the two sections are tangent to each
other in one of the double section), and the other fibers are over
$[u_i:1]$, where $u_i^3=1$. The corresponding fibers pass through the
intersections of the sections of ${\mathcal {O}}(4)$, hence each such
fiber passes through two basepoints in the pencil in ${\mathbb
 {F}}_2$, therefore both curves are in the pencil.

In a similar manner, when starting with the sections $u^4$ and
$uv^2(u+v)$, we get a pencil giving a fibration with an $I_0^*$ fiber,
a type $IV$ fiber and an $I_1$ fiber, hence (again by the
classification) there is a single further singular fiber which is an
$I_1$. (The fibers are over the points $[u_i:1]$ where $u_i$ solves
$u^3=u+1$.)
These two examples exhaust the cases having an $I_0^*$ fiber
and a~type~$IV$ fiber next to it.
\end{Example}

\begin{Example}[fibration with an $I_3$ fiber]\label{exam:PeldakI3}
Consider now two sections $\sigma _1$, $\sigma _2$ of the ruling
on~${\mathbb {F}}_2$ intersecting each other in~$F_1$, and in a further
point which is not on the chosen fibers $F_i$ ($i=2,3,4$).
Taking this curve as~$C_0$, the resulting pencil will have an $I_3$ fiber
next to the $I_0^*$ fiber. In~$\CP{2}$ we consider $S_0$ which is again a
union of three lines $L_1$, $L_2$, $L_3$, with $L_1$ passing through~$P$ and
$L_2$, $L_3$ intersecting each other in a generic point $P'$ (which is
not on $L_1$ and on any $\ell _i$); see Fig.~\ref{fig:I3}.
\begin{figure}[t] \centering
\includegraphics[scale=0.8]{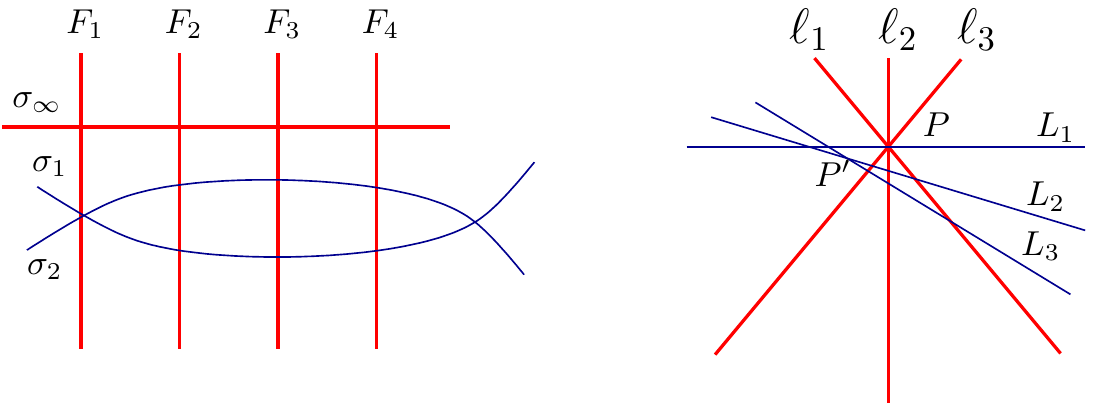}
\caption{The diagram depicts two pencils (one in ${\mathbb {F}}_2$
 and the other one in $\CP{2}$) giving fibrations with an~$I_0^*$ fiber and an~$I_3$ fiber. The conventions are the same as
 earlier.}\label{fig:I3}
\end{figure}
Fibrations admitting an $I_0^*$ and an $I_3$ fiber on the rational
elliptic surface have further singular fibers; the possibilities are: either a
cusp and a fishtail, or three fishtails. The second case is the
generic one. Let us see how to construct such a fibration explicitly.
For this purpose, we take
\begin{gather}\label{eq:SIGMA1}
 \sigma_1 = (u+v)^2 v (u - v), \qquad \sigma_2 = u^2 v^2.
\end{gather}
Then $\sigma_1$ is not a complete square, but has two equal zeroes,
so the equation $\zeta \otimes \zeta = \sigma_1$ defines a double section with
a nodal point over $[u:v] = [-1:1]$, giving rise to a fiber of type $I_1$.
On the other hand, $\sigma_2$ is a complete square with zeroes
at $[u:v] = [0:1]$ and $[u:v] = [1:0]$.
The basepoints are determined by the equation
\begin{gather*}
 (u+v)^2 v (u - v) = u^2 v^2,
\end{gather*}
which has roots $[u:v] = [1:0]$ and $[u:v] = [u_i:1]$ where $u_1$, $u_2$, $u_3$ are the roots of
\begin{gather*}
 (u+1)^2(u-1) - u^2=0.
\end{gather*}
As one of the roots $[1:0]$ of $\sigma_2$ is at the same time a basepoint while its other root
$[0:1]$ is not, $\sigma_2$ gives rise to a fiber of type $I_3$.
Now, let the parameter of the pencil be denoted $\lambda = [\lambda_1 : \lambda_2] \in\CP1$.
Then, the curve corresponding to $\lambda$ in the pencil has equation
\begin{gather*}
 (\lambda_1 + \lambda_2) \zeta \otimes \zeta = \sigma_{\lambda},
\end{gather*}
where
\begin{gather*}
 \sigma_{\lambda} = \lambda_1 \sigma_1 + \lambda_2 \sigma_2.
\end{gather*}
The parameter value $\lambda = [-1:1]$ corresponds to a degenerate
 curve consisting of the fibers containing the basepoints (and the
 section at infinity in ${\mathbb {F}}_2$ with multiplicity $2$), and
 gives rise to the fiber of type $I_0^*$. From now on we assume that
 $\lambda \neq [-1:1]$, and the equation of the corresponding curve
 can be written as
\begin{gather*}
 \zeta \otimes \zeta = \frac{\sigma_{\lambda}}{\lambda_1 + \lambda_2}.
\end{gather*}
Expressing the right hand side of this equation using equation~\eqref{eq:SIGMA1},
and taking $v=1$ we get
\[
\lambda _1u^3+(\lambda _1+\lambda _2)u^2-\lambda _1u-\lambda _1.
\]
The discriminant of this polynomial is given by
\begin{gather*}
 \lambda_1 \lambda_2 \frac{32 \lambda_1^2 + 13 \lambda_1 \lambda_2 + 4 \lambda_2^2}{(\lambda_1 + \lambda_2)^4}.
\end{gather*}
The parameter values $\lambda_1 = 0$ and $\lambda_2 = 0$ define $\sigma_2$ and $\sigma_1$ respectively,
hence we have already treated them. (Remember that $\lambda_1$, $\lambda_2$ may not simultaneously vanish.)
Consider the roots $\lambda^+=[\lambda _1^+:\lambda _2^+], \lambda^- =[\lambda _1^-:\lambda _2^-] \in\CP1$ of
\begin{gather*}
 32 \lambda_1^2 + 13 \lambda_1 \lambda_2 + 4 \lambda_2^2;
\end{gather*}
a straightforward check shows that $\lambda^+ \neq \lambda^-$ and that
\begin{gather*}
 \lambda^+, \lambda^- \in\C\setminus \{ [0:1] , [-1:1] \}\subset\CP1.
\end{gather*} Then, $\sigma_{\lambda^{\pm}}$ both have a double zero, therefore
 each of the corresponding curves has at least one singular
 point. Moreover, these singular curves are different from the ones
 we found earlier. By the classification, the corresponding curves in
 the fibration are both of type $I_1$, and we are done.

Now, let us turn our attention to finding an example of the non-generic case,
i.e., a cuspidal curve beside $I_0^*$ and $I_3$.
Using our earlier method of generating double sections
from homogeneous degree-4 polynomials, we proceed as follows. Take the
polynomials $u^2v^2$ and $u(u+v)^3$, and consider the fibers $F_1,
\ldots , F_4$ over the points $[u_i:1]\in \CP{1}$ for which
$u^2=u(u+1)^3$ holds. These are the fibers passing through the
intersection points of the sections in ${\mathcal {O}}(4)$, hence of
pairs of basepoints in ${\mathbb {F}}_2$. The second curve is a
cuspidal curve, hence the resulting fibration will have an
$I_0^*$ fiber, together with an $I_3$ fiber and a cusp (and a further
$I_1$ fiber, as dictated by the classification).
\end{Example}

\begin{Example}[fibration with $III$]\label{exam:PeldakIII}
The curve $C_0$ in ${\mathbb {F}}_2$ is a union of two sections, which
are tangent to each other in a point not on the chosen fibers $F_i$
($i=1,2,3,4$). (Alternatively, if we get a fibration with a
double-section which has a cusp point on the fiber~$F_4$~-- as shown in
Fig.~\ref{fig:III} -- we get a fibration with an $I_0^*$ fiber and a~type~$III$ fiber.) The curve $S_0$ in the pencil in~$\CP{2}$ now
consists of two curves, a line $L_1$ passing through $P$ and a quadric
$Q$ which does not pass through $P$ and it is tangent to~$L_1$, as
shown by the right diagram of Fig.~\ref{fig:III}.
\begin{figure}[t] \centering
\includegraphics[scale=0.8]{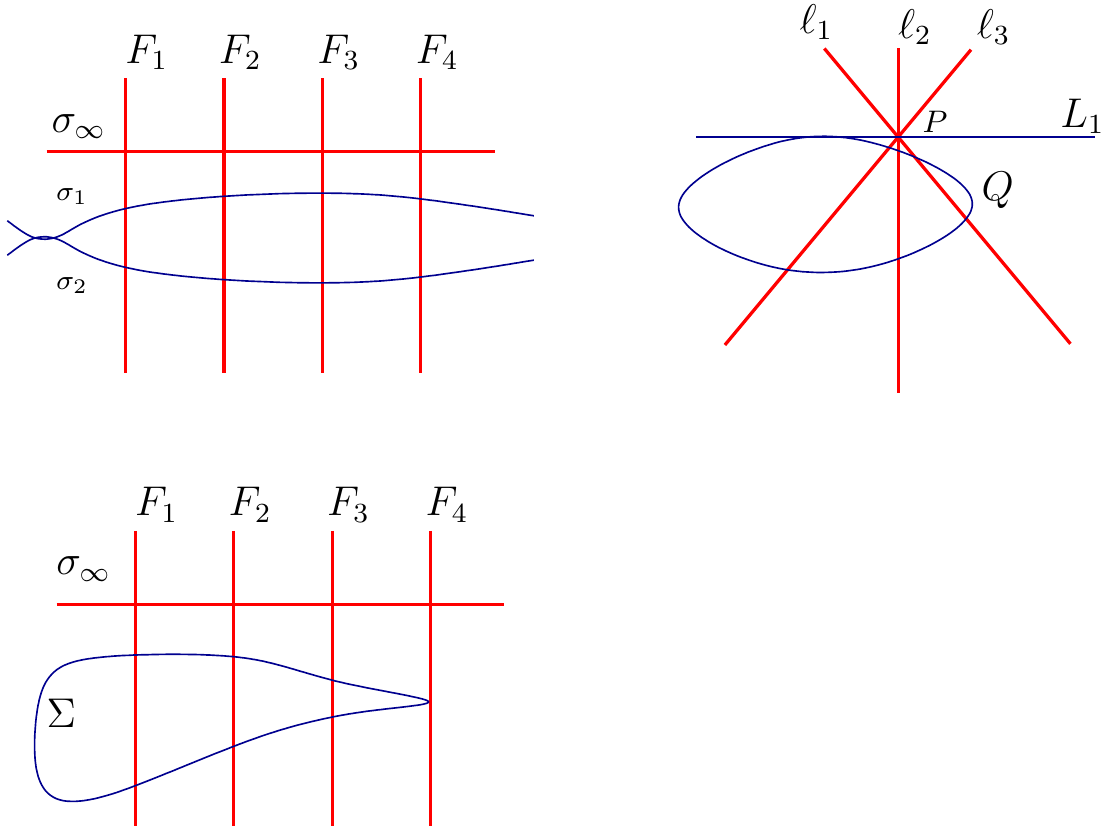}
\caption{Continuing with the same conventions as above, the diagrams show pencils defining fibrations with an $I_0^*$ fiber and a type $III$ fiber. In this case we describe two pencils in ${\mathbb {F}}_2$, one having eight, the other one seven basepoints.}\label{fig:III}
\end{figure}
There are four possibilities of the further fibers (listed in the second row of Table~\ref{tab:singfibers}): it can be (a) another type $III$ fiber, (b) $I_2$ and $I_1$, (c) $II$ and $I_1$ and finally (d) three $I_1$'s.

By taking the degree-4 polynomials $u^3v$ and $uv^3$ and the fibers over $[0:1], [\pm 1 :1], [1:0]\allowbreak \in \CP{1}$, the pencil will give rise to a fibration with an $I_0^*$ fiber and two type $III$ fibers. Similarly, by taking $u^4$ and $(u+v)v^3$ and the fibers over $[u_i:1]$ with $u_i$ solving $u^4=(u+1)$, we get a~pencil giving a~fibration with an $I_0^*$, a~type~$III$ and a cusp (type~$II$) fibers, and hence necessarily a~further~$I_1$ fiber. In a similar manner, starting with $u^3v$ and $uv^2(u+v)$, and taking the fibers over $[0:1], [1:0]\in \CP{1}$ and $[u_i:1]\in \CP{1}$ with $u_i$ solving $u^2=u+1$, we get a pencil giving a fibration with an $I_0^*$ fiber, a~type $III$ and an $I_2$ fibers (and therefore a further $I_1$). In the generic situation we have three $I_1$ fibers next to $I_0^*$ and $III$. This case can be shown to exist by the same method as the generic case in Example~\ref{exam:PeldakI3} has been handled. Specifically, setting $\sigma_1 = u^4$ (giving rise to a curve of type~$III$ with singularity over $[0:1]$) and $\sigma_2 = (u-v)^2 v (u+v)$ (giving rise to a curve of type $I_1$ with singularity over $[1:1]$), the generic curve in the pencil again reads as
\begin{gather*}
 \sigma_{\lambda} = \lambda_1 \sigma_1 + \lambda_2 \sigma_2.
\end{gather*}
A computation shows that the discriminant of $\sigma_{\lambda}$ is given by
\begin{gather*}
 \Delta = \lambda_1 \lambda_2^3 \big(256 \lambda_1^2 - 107 \lambda_1 \lambda_2 - 32 \lambda_2^2 \big).
\end{gather*}
The quadratic term in the parentheses admits two distinct roots in $\C^{\times}$, so there exist two values $\lambda^{\pm} = \big[\lambda_1^{\pm} : \lambda_2^{\pm}\big]$ such that $\sigma_{\lambda^{\pm}}$ is singular. By the classification, in the fibration these singular curves must be of type $I_1$. Hence, together with $[\lambda_1 : \lambda_2] = [0:1]$ and $[\lambda_1 : \lambda_2] = [1:0]$ we have a curve of type $III$ and three curves of type $I_1$.
\end{Example}

\begin{Example}[fibration with an $I_2$ fiber] To get a fibration with an $I_2$ fiber, we consider two sections $\sigma _1$, $\sigma _2$ intersecting each other in two distinct points, none of them on the chosen fibers $F_i$ ($i=1,2,3,4$). Alternatively, we can choose a nodal curve with the node on the fiber $F_4$, shown on Fig.~\ref{fig:I2}.
\begin{figure}[t] \centering
\includegraphics[scale=0.8]{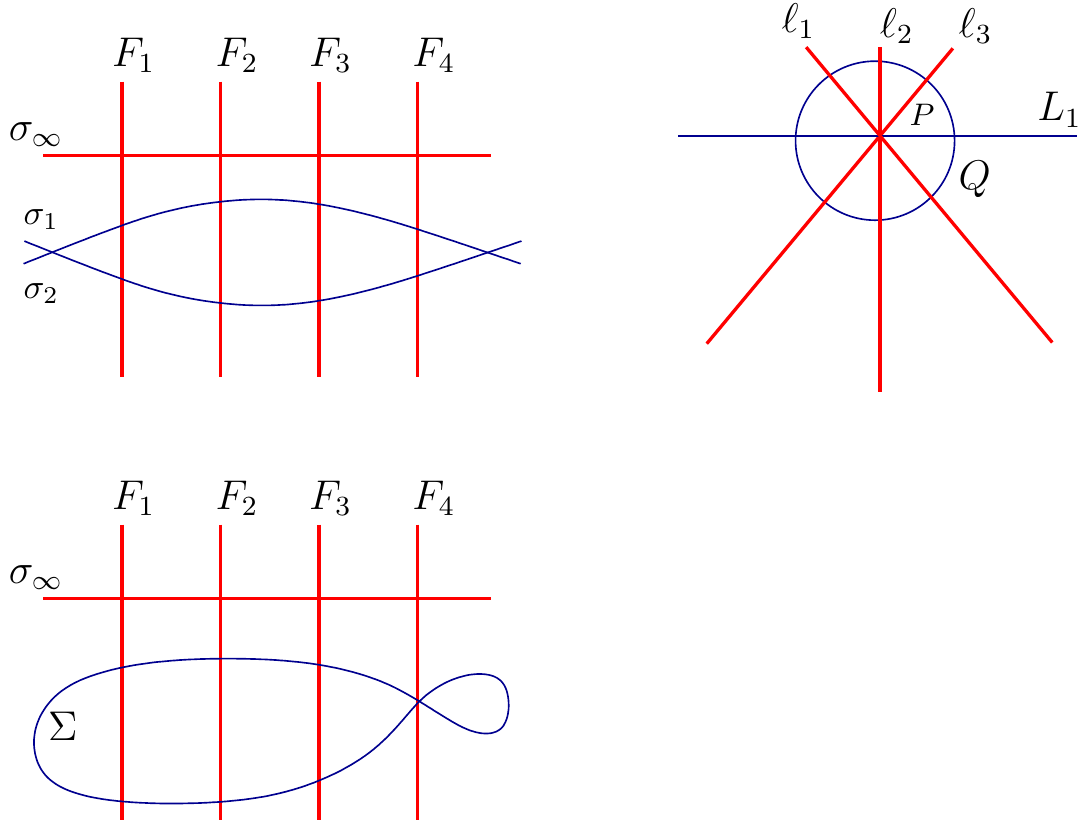}
\caption{The pencils give rise to fibrations with an $I_0^*$ fiber and an $I_2$ fiber.}\label{fig:I2}
\end{figure}
To get such a fibration from a pencil in $\CP{2}$, take $L_1$ as earlier, together with a~quadric, which does not pass through $P$ and it is transverse to~$L_1$, see Fig.~\ref{fig:I2}.

There are five cases for the further singular fibers, as shown
in the third row of Table~\ref{tab:singfibers}: next to $I_0^*$ and
$I_2$ we can have: (a) two further $I_2$, (b) an $I_2$ and two $I_1$,
(c) two type $II$, (d)~a~type~$II$ and two~$I_1$ and finally (e) four $I_1$.
We will discuss these cases below.
\end{Example}

\begin{Example}[three $I_2$ fibers]
Consider an elliptic pencil on ${\mathbb {F}}_2$ whose eight basepoints are
all distinct, pairwise lying on the fibers over $0$, $1$, $t$, $\infty $ of the
ruling. Assume in addition that there exists a section $\sigma_1$ of
${\mathbb {F}}_2$ passing through four of these basepoints and a
section $\sigma_2$ passing through the other four basepoints
(see the upper left diagram of Fig.~\ref{fig:I2}).

If there exist two sections $\sigma_1$, $\sigma_2$ satisfying the above
conditions, then in the corresponding fibration they give rise to a
curve of type $I_2$. In order to find a fibration with three
$I_2$ fibers, it is thus sufficient to find values $a_{\pm}$, $b_{\pm}$, $c_{\pm}$, $d_{\pm}$ such that there exist simultaneously
\begin{itemize}\itemsep=0pt
 \item a pair of sections $\sigma_1$, $\sigma_2$ such that $\sigma_1$
 passes through $a_+$, $b_+$, $c_+$, $d_+$ and $\sigma_2$ passes through
 $a_-$, $b_-$, $c_-$, $d_-$;
 \item a pair of sections $\sigma_3$, $\sigma_4$ such that $\sigma_3$
 passes through $a_-$, $b_+$, $c_-$, $d_+$ and $\sigma_4$ passes through $a_+$, $b_-$, $c_+$, $d_-$;
 \item a pair of sections $\sigma_5$, $\sigma_6$ such that $\sigma_5$
 passes through $a_+$, $b_+$, $c_-$, $d_-$ and $\sigma_6$ passes through $a_-$, $b_-$, $c_+$, $d_+$.
\end{itemize}
Each of these pairs of sections then give rise to different curves of type $I_2$ in the fibration coming
from the pencil defined by the corresponding set of basepoints, and by an Euler characteristic computation
this fibration has no other singular fibers.

A necessary and sufficient condition for the existence of such a pair of sections can be found as follows. In the trivialization {of $K_{\CP1}(D)$ over $U$}
discussed {in~\eqref{eq:w-coord}} %\peter{Section~\ref{ssec:param_geo}},
for $1\leq j \leq 2$ the section~$\sigma_j$ of~${\mathbb {F}}_2$ is of the form
\begin{gather*}
 \sigma_j (u) = \alpha_j u^2 + \beta_j u + \gamma_j
\end{gather*}
for some $\alpha_j ,\beta_j ,\gamma_j \in\mathbb{C}$. For simplicity assume that $t = -1$.
Then, $\sigma_1$ passes through the basepoints if and only if the following equations hold
\begin{gather*}
 \gamma_1 = a_+, \\
 \alpha_1 + \beta_1 + \gamma_1 = b_+, \\
 \alpha_1 - \beta_1 + \gamma_1 = c_+, \\
 \alpha_1 = d_+,
\end{gather*}
where $a_{\pm}$, $b_{\pm}$, $c_{\pm}$ are the $w$-coordinates of
the basepoints over $0$, $1$, $t=-1$ respectively and $d_{\pm}$ are the~$z$ coordinates over~$\infty$.
This linear system
in three variables $\alpha_1$, $\beta_1$, $\gamma_1$ has a solution if and
only if
\begin{gather*}
 2(a_+ + d_+ ) = b_+ + c_+.
\end{gather*}
Similarly, the existence of $\sigma_2$ is equivalent to
\begin{gather*}
 2(a_- + d_- ) = b_- + c_-.
\end{gather*}

We may repeat the above argument
to find conditions for the existence of $\sigma_3$, $\sigma_4$, $\sigma_5$, $\sigma_6$. These conditions, together with the ones found above for $\sigma_1$, $\sigma_2$, form a system of $6$ homogeneous linear equations in the $8$ variables $a_{\pm}$, $b_{\pm}$, $c_{\pm}$, $d_{\pm}$. Completing this system with the degree formula
\begin{gather}\label{eq:degree}
 a_+ + b_+ + c_+ + d_+ + a_- + b_- + c_- + d_- = 0
\end{gather}
results in a system of $7$ homogeneous linear equations in~$8$ variables. In view of the assumption that the basepoints are all distinct, in order to get a fibration with three~$I_2$ fibers we merely need to show that the above linear system admits generic solutions, namely solutions such that
\begin{gather}\label{eq:generic}
 a_+ \neq a_-, \qquad b_+ \neq b_-, \qquad c_+ \neq c_-, \qquad d_+ \neq d_-.
\end{gather}
Now, \looseness=-1 the coefficient matrix with respect to the variables $a_+, a_-, \ldots, d_+, d_-$ of the system of equations governing the existence of
$\sigma_1, \ldots , \sigma_6$ complemented by the degree formula~\eqref{eq:degree} reads as
\begin{gather*}
 \begin{pmatrix}
 2 & 0 & -1 & 0 & -1 & 0 & 2 & 0 \\
 0 & 2 & 0 & -1 & 0 & -1 & 0 & 2 \\
 0 & 2 & -1 & 0 & 0 & -1 & 2 & 0 \\
 2 & 0 & 0 & -1 & -1 & 0 & 0 & 2 \\
 2 & 0 & -1 & 0 & 0 & -1 & 0 & 2\\
 0 & 2 & 0 & -1 & -1 & 0 & 2 & 0\\
 1 & 1 & 1 & 1 & 1 & 1 & 1 & 1
 \end{pmatrix}.
\end{gather*}
One numerically checks that this matrix is of rank $5$, hence it
 admits a three-dimensional family of solutions. As a matter of
 fact, the last row of the coefficient matrix is a linear combination
 of the first six rows, and the corresponding last entry in the
 extended matrix of the linear system of equations is~$0$, i.e., the
 degree formula automatically holds once there exists three fibers of
 type~$I_2$ in the given configuration; consequently, the degree
 formula does not effect the solvability of the system. In addition,
 one sees that $c_-$, $d_+$, $d_-$ may be chosen as parameters and the
 other variables may be expressed in terms of these as
\begin{gather*}
 a_+ = - d_-, \qquad a_- = - d_+, \qquad b_+ = - c_-, \qquad b_- = 2d_- - 2d_+ - c_- , \\ c_+ = -2d_- + 2d_+ + c_- .
\end{gather*} It is therefore sufficient to choose any $d_- \neq d_+$ to get
 solutions satisfying~\eqref{eq:generic}. Finally, we note that a
 similar attempt to find four $I_2$ fibers fails because the
 corresponding linear system has two free parameters $d_+$, $d_-$ and
 the other variables can be expressed as
\begin{gather*}
 a_+ = - d_-, \quad a_- = - d_-, \quad \ldots.
\end{gather*}
In particular, there exist no solution satisfying \eqref{eq:generic}.
\end{Example}

\begin{Example}[three $I_2$ fibers]
Another, less explicit construction of the same type of fibration can
be given as follows: consider a nodal double section of ${\mathbb
 {F}}_2$ with the property that the node is on $F_3$ while $F_1$ and
$F_4$ are tangent to the nodal double section, see
Fig.~\ref{fig:3I2}. (It is easy to see from the Hurwitz-formula
that a nodal double section can have at most two tangent fibers.)
\begin{figure}[t] \centering
\includegraphics[scale=0.8]{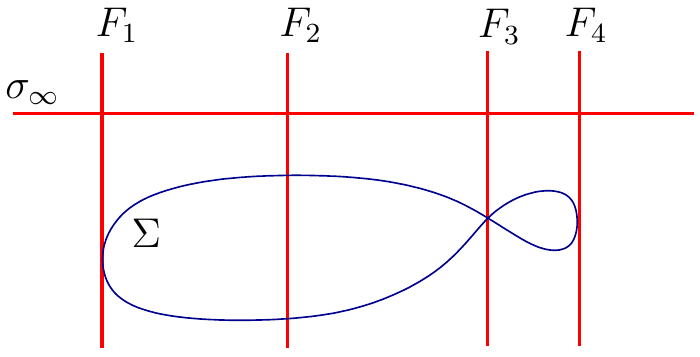}
\caption{A pencil which gives rise to fibrations with an $I_0^*$ fiber and three $I_2$ fibers.}\label{fig:3I2}
\end{figure}
In this case an $I_2$ fiber is originated from the blow-up of the node of
the chosen double section. On the other hand, when we blow up the basepoints
where our chosen double section is tangent to~$F_1$ and~$F_4$, we see that
disjoint rational $(-2)$-curves appear in fibers, hence by the classification
the further fibers should be two~$I_2$ fibers. (Indeed, the pencil contains
two further nodal curves, their nodes are on~$F_1$ and~$F_4$ and are tangent
to two chosen fibers.)
\end{Example}

If only $F_1$ is tangent to our chosen curve (see Fig.~\ref{fig:I2var}),
\begin{figure}[t] \centering
\includegraphics[scale=0.8]{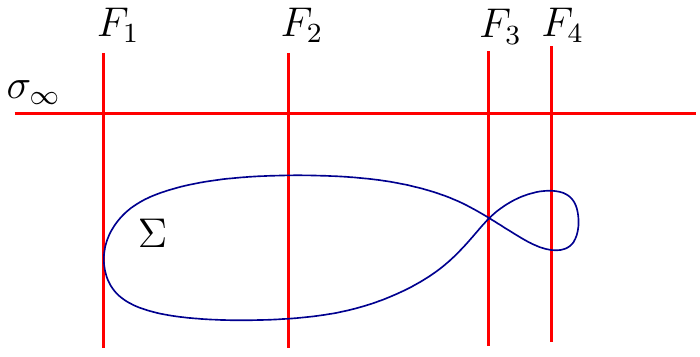}
\caption{A pencil which gives rise to fibrations with an $I_0^*$ fiber, an $I_2$ fiber and further $I_2$ or type~$III$ fibers.}\label{fig:I2var}
\end{figure}
then the fibration (next to $I_0^*$ and $I_2$) has either a further $I_2$
(and also two $I_1$) or a type $III$ (and a further $I_1$). The outcome
depends on the fact whether the curve with singularity on $F_1$
is nodal (giving a further $I_2$) or cuspidal (providing a type $III$ fiber).
Clearly the first is the generic case among these examples.
In the next examples we will give constructions of these cases.

\begin{Example}[two $I_2$ fibers] By choosing appropriate homogeneous polynomials, we can arrange that we have two $I_2$ fibers (and two further $I_1$ fibers). Indeed, by taking $u^2v(u+v)$ and $uv^2(u+2v)$, for the first curve we have a node and a tangency (for the second a tangency and a~node) at the fibers over $[0:1]$ and $[1:0]$, and by choosing the fibers over $[u_i:1]$ with $u_i$ solving the equation $u(u+1)=u+2$, we get the desired pencil and hence fibration. Indeed, next to $I_0^*$ and the two $I_2$ fibers, no further fiber can contain a $(-2)$-curve (since the pencil cannot contain sections, and no further singular point can be on the fibers of the ruling).

A more specific construction may be given along the lines of Example~\ref{exam:PeldakI3},
choosing $\sigma_1 = (u-v)^2 (u+2v)^2$, $\sigma_2 = u^2v^2$, with the corresponding discriminant
\begin{gather*}
 \Delta = 64 \lambda_1^2 \lambda_2^2 \big(81 \lambda_1^2 -14 \lambda_1 \lambda_2 + \lambda_2^2\big)
\end{gather*}
admitting two distinct roots (other than at $[\lambda_1 : \lambda_2] = [0:1]$ and $[1:0]$).
\end{Example}

\begin{Example}[one $I_2$ fiber]A fibration with an $I_2$ and a cusp fiber $II$ (next to the $I_0^*$
fiber) and by classification having either a further $II$ or two
further $I_1$ fibers) can be given by taking the homogeneous degree-4
polynomials $u^2v^2$ and $(u+v)(u+xv)^3$ with parameter $x\neq 0, 1$
(the first giving an $I_2$ fiber, the second a cusp curve with cusp
over $[-x:1]$), where the fibers of the ruling ${\mathbb {F}}_2\to
\CP{1}$ are chosen over $[0:1]$ and over $[u_i:1]$ with $u_i$ solving
$u^2=(u+1)(u+x)^3$. The two cases for the additional fibers can both
occur. In the pencil generated by the above two singular curves, the
discriminant is
\begin{gather*}
-\lambda _1^2 \lambda _2^2 x^3 \big({-}16 \lambda _1^2+4 \lambda _1 \lambda _2+ 27 \lambda _2^2 x^5-108 \lambda _2^2 x^4+162
 \lambda _2^2 x^3+4 \lambda _1 \lambda _2 x^3-108 \lambda _2^2 x^2 \\
\qquad{} -132 \lambda _1 \lambda _2 x^2+27 \lambda _2^2 x-132 \lambda _1 \lambda _2 x\big).
\end{gather*}
Setting $\lambda _2=1$, the term in the parentheses has discriminant
in $\lambda_1$ equal to
\begin{gather*}
	16 \big(x^2+14 x+1\big)^3.
\end{gather*}
The roots of this polynomial are distinct from $0$, $1$. If $x$ is a
root of this polynomial, then the pencil has only one further singular
fiber which is then necessarily a cusp. Otherwise the pencil has two
further singular fibers which are fishtails.

Finally, the generic case is when next to $I_0^*$ and the $I_2$ fiber
there are four $I_1$ fibers. To get an example with these singular
fibers, pick $\sigma_1 = (u-v)^2 (u+v)^2$, $\sigma_2 = u^2 v (u+2v)$,
giving rise to fibers of type $I_2$, $I_1$ respectively. Then the
generic curve in the pencil has discriminant
\begin{gather*}
 \Delta = \lambda_1 \lambda_2^2 \big(768 \lambda_1^3 - 832 \lambda_1^2 \lambda_2 + 325 \lambda_1 \lambda_2^2 -32 \lambda_2^3\big).
\end{gather*}
The cubic polynomial in parentheses admits three distinct roots, each providing a value of $[\lambda_1 : \lambda_2]$
for which the corresponding curve in the fibration is singular. By the classification, these singular curves are of type $I_1$.
\end{Example}

It is easy to see that if the fibration contains (next to $I_0^*$) a~fiber which is $I_4$, $I_3$, $I_2$ or type $IV$ or $III$ (their common property
being that they contain rational $(-2)$-curves), then the pencil in~${\mathbb {F}}_2$
either contains curves with more than one components (i.e., unions
of two sections of the ruling), or nodal/cuspidal curves with their
singular points on a chosen fiber $F_i$. These cases comprise the first three rows
of the list summarized in Table~\ref{tab:singfibers}, and were all considered above.

If no such curves are present in the pencil, then the singular curves
in the pencil are either cusps or nodal curves, with their singular
points away from the chosen fibers. There are four possibilities, listed in the
fourth row of Table~\ref{tab:singfibers}: next
to the $I_0^*$ fiber we can have (a) three cusps, (b) two cusps and two
fishtails, (c) one cusp and four fishtails and finally (d) six
fishtails. Clearly, the last case is the generic. In the following
we will discuss some cases left open above. We will only discuss
pencils in the Hirzebruch surface ${\mathbb {F}}_2$.

\begin{Example}[three or two cusp fibers]
It is relatively easy to find a pencil giving rise to a fibration with
at least two cusps (next to $I_0^*$): consider the double sections coming from the
homogeneous degree-4 polynomials $u^3v$ and $(u+xv)^3(u+v)$ with parameter $x\neq 0,1$.
If we take the fibers over the points $[u_i:1]\in \CP{1}$ where
$u_i$ solves $(u+x)^3(u+1)=u^3$, we get at least two cusps.

We are in cases (a) or (b) above. The actual case depends on the value
of $x$. Indeed the general polynomial of the pencil has discriminant
in $u$ (setting $v=1$) equal to
\begin{gather*}
-\lambda _1^2 \lambda _2^2 x^6 \big(27 \lambda _1^2+27 \lambda _2^2+54 \lambda _1 \lambda _2+27 \lambda _2^2 x^4-108 \lambda _2^2 x^3+4 \lambda _1 \lambda _2 x^3 \\
\qquad{} +162 \lambda _2^2 x^2-18 \lambda _1 \lambda _2 x^2-108 \lambda _2^2 x+216 \lambda _1 \lambda _2 x\big).
\end{gather*}
Setting $\lambda _2=1$, the term in the parentheses has discriminant in $\lambda_1$ equal to
\begin{gather*}
	16 x (x-9)^2 (x+3)^3.
\end{gather*}
The roots of this polynomial are $-3$ and $9$. If $x$ is equal to one of these values, then the pencil has only one further singular fiber which is then necessarily a cusp. Otherwise the pencil has two further singular fibers which are fishtails.
\end{Example}

\begin{Example}[one cusp fiber] In a similar manner, taking $u^3v$ and $(u-v)^2(u+v)(u+2v)$ we get a~fibration with a cusp and a fishtail fiber next to the $I_0^*$ fiber;
the discriminant of the generic curve in the pencil is given by
\begin{gather*}
 \Delta = - \lambda_1 \lambda_2^2 \big(108 \lambda_1^3 + 320 \lambda_1^2 \lambda_2 + 1827 \lambda_1 \lambda_2^2 + 864 \lambda_2^3\big).
\end{gather*}
As the cubic in parentheses admits three distinct roots, the remaining fibers are all $I_1$ fibers, giving the example of a cusp together with four fishtails.
\end{Example}

\begin{Example}[six fishtails] The generic case in this situation is simply six $I_1$ fibers next to~$I_0^*$.
To construct an example, set $\sigma_1 = (u-v)^2 (u+v) (u+3v)$, $\sigma_2 = u^2 v (u+2v)$, both giving rise to a type $I_1$ fiber.
The discriminant of the generic curve in the pencil reads as
\begin{gather*}
 \Delta = - \lambda_1 \lambda_2 \big(24576 \lambda_1^4 - 8272 \lambda_1^3 \lambda_2 + 3768 \lambda_1^2 \lambda_2^2 -517 \lambda_1 \lambda_2^3 + 96 \lambda_2^4\big).
\end{gather*} It can be shown that the quartic polynomial in parentheses has non-vanishing discriminant, therefore $\Delta$ has four distinct roots $[\lambda_1 : \lambda_2]$, giving rise to four further fibers of type $I_1$.
\end{Example}

\begin{proof}[Proof of Theorem~\ref{thm:I0*osszeseset}]
The combination of the examples given in this section provide the proof of the theorem.
\end{proof}

\subsection*{Acknowledgments}
The authors acknowledge support of NKFIH through the \emph{\'Elvonal (Frontier)} program, with grant KKP126683. The third author was also supported by the grant NKFIH 120697.
The authors would like to thank the anonymous referees for many helpful comments, suggestions and corrections.

\pdfbookmark[1]{References}{ref}
\LastPageEnding

\end{document}